\documentclass[10pt, reqno, oneside]{amsart}
\usepackage[foot]{amsaddr}
\usepackage[margin=1in]{geometry}
\usepackage{ytableau}
\usepackage{todonotes}
\usepackage[normalem]{ulem}
\usepackage{subcaption}

\usepackage{amssymb,amsmath,amsthm,amscd,epsf,latexsym,verbatim,graphicx,amsfonts,hyperref,cleveref,xcolor,wasysym}
\usepackage{mathtools}
\usepackage{amsfonts}
\usepackage{mathrsfs}

\usepackage{thmtools} 
\usepackage{thm-restate}

\usepackage{wrapfig}
\usepackage{cutwin}
\usepackage{tikz}
\usepackage{tkz-graph}
\usepackage{caption}
\usepackage{float}
\usetikzlibrary{shapes,snakes, positioning}
\usetikzlibrary{arrows}
\usetikzlibrary{shapes.misc}
\usepackage{rotating}
\usetikzlibrary{decorations.markings}

\usepackage{array}  
\usepackage{longtable} 
\usepackage{booktabs}
\newcolumntype{C}{>{$}c<{$}} 
\tikzset{cross/.style={cross out, draw=black, fill=none, minimum size=2*(#1-\pgflinewidth), inner sep=0pt, outer sep=0pt}, cross/.default={2pt}}

\newtheorem{theorem}{Theorem}[section]

\newtheorem{lemma}[theorem]{Lemma}
\newtheorem{proposition}[theorem]{Proposition}

\theoremstyle{definition}

\theoremstyle{definition}

\theoremstyle{definition}
\newtheorem{remark}[theorem]{Remark}
\theoremstyle{definition}

\theoremstyle{definition}
\newtheorem{definition}[theorem]{Definition}
\theoremstyle{definition}

\theoremstyle{definition}

\theoremstyle{definition}

\newtheoremstyle{named}{}{}{\itshape}{}{\bfseries}{.}{.5em}{\thmnote{#3 }#1}
\theoremstyle{named}

\newcommand{\ZZ}{\mathbb{Z}}
\newcommand{\RR}{\mathbb{R}}

\newcommand{\QQ}{\mathbb{Q}}
\newcommand{\CC}{\mathbb{C}}
\newcommand{\NN}{\mathbb{N}}
\newcommand{\TT}{\mathbb{T}}
\newcommand{\EE}{\mathbb{E}}

\newcommand{\STS}{\hbox{\sf STS}}

\newcommand{\calH}{\mathcal{H}}
\newcommand{\Hol}{\hbox{\sf Hol}}

\newcommand{\Var}{\hbox{\rm Var}}
\newcommand{\Tr}{\hbox{\rm Tr}}
\newcommand{\Id}{\hbox{\rm Id}}
\newcommand{\Irr}{\hbox{\rm Irr}}
\newcommand{\Par}{\hbox{\rm Par}}
\newcommand{\CF}{\hbox{\rm CF}}
\newcommand{\ch}{\hbox{\rm ch}}
\newcommand{\cyc}{\hbox{\rm cyc}}

\newcommand{\cont}{\hbox{\rm cont}}

\newcommand{\calK}{\mathcal{K}}
\newcommand{\Unif}{\hbox{\rm Unif}}

\newcommand{\SL}{\mathrm{SL}}

\DeclarePairedDelimiter{\ideal}{\langle}{\rangle}
\DeclarePairedDelimiter\floor{\lfloor}{\rfloor}

\title{A model for horizontally restricted random square-tiled surfaces}

\author{Nick Fitzhugh}
\email{nicholas.j.fitzhugh@gmail.com}

\author{Aaron Schondorf}
\email{amschondorf@gmail.com}

\author{Sunrose Shrestha}
\address{Department of Mathematics, Carleton College, Northfield, MN}
\email{sshrestha@carleton.edu}

\author{Sebastian Vander Ploeg Fallon}
\address{Department of Physics, Cornell University, Ithaca, NY}
\email{sjv56@cornell.edu}

\author{Thomas Zeng}
\address{Department of Computer Science, University of Wisconsin at Madison, Madison, WI}
\email{tpzeng@wisc.edu}

\begin{document}


\begin{abstract}
A square-tiled surface (STS) is a (finite, possibly branched) cover of the standard square-torus with possible branching over exactly 1 point. Alternately, STSs can be viewed as finitely many axis-parallel squares with sides glued in parallel pairs. After a labelling of the squares by $\{1, \dots, n\}$, we can describe an STS with $n$ squares using two permutations $\sigma, \tau \in S_n$, where $\sigma$ encodes how the squares are glued horizontally and $\tau$ encodes how the squares are glued vertically. Hence, a previously considered natural model for STSs with $n$ squares is $S_n \times S_n$ with the uniform distribution. We modify this model to obtain a new one: We fix $\alpha \in [0,1]$ and let $\calK_{\mu_n}$ be a conjugacy class of $S_n$ with at most $n^\alpha$ cycles. Then $\calK_{\mu_n} \times S_n$ with the uniform distribution is a model for STSs with $n$ squares and restricted horizontal gluings. Since horizontal cycles of the $\sigma$ permutation are related to the number of maximal horizontal cylinders, this new model serves as a random model for STSs with at most $n^\alpha$ maximal horizontal cylinders. We deduce the asymptotic (as $n$ grows) number of components, genus distribution, most likely stratum and set of holonomy vectors of saddle connections for random STSs in this new model. 
\end{abstract}

\maketitle

\section{Introduction}

\begin{wrapfigure}{r}{0.4\textwidth}
\resizebox{6.5cm}{!}{
\begin{tikzpicture}
    [sq/.style=
  {shape=regular polygon, regular polygon sides=4, draw, minimum width=1.414cm}, rect/.style={shape=regular polygon, regular polygon sides=4, draw, minimum width=1pt}]
   \node[sq, fill=gray] at (2,1){$7$};
  \node[sq, fill=gray!25] at (0,0){$4$};
  
   \node[sq, fill=gray!25] at (1,0){$5$};
   \node[sq, fill=gray!25] at (-1,0){$3$};
   \node[sq, fill=gray] at (1,1){$6$};
  
   \node[sq, fill=gray!50] at (-1,1){$2$};
   \node[sq, fill=gray!50] at (-2,1){$1$};
   \node[sq, fill=gray] at (1,2){$8$};
   \node[sq, fill=gray] at (2,2){$9$};

    \foreach \x in {3,5}{
   \draw[fill=black] (\x*45:0.707cm) circle [radius=2pt];
   \draw[shift=({-2,1}),fill=black] (\x*45:0.707cm) circle [radius=2pt];}
   \draw[shift=({-1,1}),fill=black] (1*45:0.707cm) circle [radius=2pt];
    \foreach \x in {1,7}{
    \draw (\x*45:0.607cm) rectangle (\x*45:0.807cm);}
    \draw[shift=({1,2})](3*45:0.607cm) rectangle (3*45:0.807cm);
    \draw[shift=({2,2})] (1*45:0.607cm) rectangle (1*45:0.807cm);
    \draw[shift=({2,1})] (7*45:0.607cm) rectangle (7*45:0.807cm);
    \foreach \x in {3, 5}{
    \draw[shift=({-1,0})] (\x*45:0.707cm) circle [radius=2.5pt];
    }
    \foreach \x in {1, 7}{
    \draw[shift=({1,0})] (\x*45:0.707cm) circle [radius=2.5pt];
    }
    \draw[shift=({-1,1})] (3*45:0.707cm) circle [radius=2.5pt];
  \draw[shift=({1,2})] (1*45:0.707cm) circle [radius=2.5pt];

    \draw[thin] (0,0.35) -- ++(0,0.3);
    \draw[thin] (0,-0.35) -- ++(0,-0.3);

    \draw[thin] (-1.05,1.35) -- ++(0,0.3);
    \draw[thin] (-1.05,-0.35) -- ++(0,-0.3);
    \draw[thin] (-0.95,1.35) -- ++(0,0.3);
    \draw[thin] (-0.95,-0.35) -- ++(0,-0.3);

    \draw[thin] (1,1.35+1) -- ++(0,0.3);
    \draw[thin] (1,-0.35) -- ++(0,-0.3);
    \draw[thin] (1.05,1.35+1) -- ++(0,0.3);
    \draw[thin] (1.05,-0.35) -- ++(0,-0.3);
    \draw[thin] (0.95,1.35+1) -- ++(0,0.3);
    \draw[thin] (0.95,-0.35) -- ++(0,-0.3);

  \draw[thin] (-1.5,0.1) -- ++(-0.15,-0.15);
  \draw[thin] (-1.5,0.1) -- ++(0.15,-0.15);
  \draw[thin] (1.5,0.1) -- ++(-0.15,-0.15);
  \draw[thin] (1.5,0.1) -- ++(0.15,-0.15);


  \draw[thin] (-2.5,0.9) -- ++(-0.15,-0.15);
  \draw[thin] (-2.5,0.9) -- ++(0.15,-0.15);
  \draw[thin] (-2.5,1.1) -- ++(-0.15,-0.15);
  \draw[thin] (-2.5,1.1) -- ++(0.15,-0.15);
  
  \draw[thin] (-0.5,0.9) -- ++(-0.15,-0.15);
  \draw[thin] (-0.5,0.9) -- ++(0.15,-0.15);
  \draw[thin] (-0.5,1.1) -- ++(-0.15,-0.15);
  \draw[thin] (-0.5,1.1) -- ++(0.15,-0.15);

\draw[thin] (0.5,1.1) -- ++(-0.15,-0.15);
  \draw[thin] (0.5,1.1) -- ++(0.15,-0.15);
  \draw[thin] (0.5,1) -- ++(-0.15,-0.15);
  \draw[thin] (0.5,1) -- ++(0.15,-0.15);
  \draw[thin] (0.5,0.9) -- ++(-0.15,-0.15);
  \draw[thin] (0.5,0.9) -- ++(0.15,-0.15);

  \draw[thin] (2.5,1.1) -- ++(-0.15,-0.15);
  \draw[thin] (2.5,1.1) -- ++(0.15,-0.15);
  \draw[thin] (2.5,1) -- ++(-0.15,-0.15);
  \draw[thin] (2.5,1) -- ++(0.15,-0.15);
  \draw[thin] (2.5,0.9) -- ++(-0.15,-0.15);
  \draw[thin] (2.5,0.9) -- ++(0.15,-0.15);

\node[] at (-2, 1.5){$\vee$};
\node[] at (-2, 0.5){$\vee$};

\node[rotate=90] at (2, 1.5+1){$\sim$};
\node[rotate=90] at (2, 0.5){$\sim$};

\node[rotate=90] at (0.5, 2){$\vee$};
\node[rotate=90] at (2.5, 2){$\vee$};

\end{tikzpicture}
}
\caption{A square-tiled surface in $\STS_9$. Matching edge/vertex decorations are glued together. With the chosen labelling, the horizontal permutation is $\sigma = (1,2)(3,4,5)(6,7)(8,9)$ and the vertical permutation is $\tau = (1)(2,3)(4)(5,6,8)(7,9)$. The surface decomposes into three maximal horizontal cylinders indicated by the different gray shadings.}
\label{fig:introSTS}
\end{wrapfigure}
\

The goal of this paper is to study a new combinatorial model for random \textbf{square-tiled surfaces} (STSs), an important class of surfaces that appear as branched covers of the standard square-torus with branching over exactly one point. Square-tiled surfaces are a special class of \textbf{translation surfaces} which are a class of metrics on two-manifolds which admit an atlas with transition functions given by Euclidean translations. Equivalently, a translation surface can also be viewed as a pair $(X, \omega)$ where $X$ is a Riemann surface and $\omega$ is a holomorphic 1-form. A third description of translation surfaces can be given via Euclidean geometry -- a translation surface in this viewpoint is a finite collection of Euclidean polygons, embedded in the plane up to translations with sides glued in parallel pairs by translations. These surfaces are flat everywhere except for finitely many points (coming from the vertices) called \textbf{cone points} where the angle is larger than $2 \pi$. Additionally, since edges are glued in pairs via translation, the angle around each vertex is an integer multiple of $2\pi$.

A square-tiled surface (STS) is a translation surface $(X, \omega)$ where $X$ covers the square-torus $\TT = \CC/\ZZ[i]$ (possibly) branched over exactly one point and $\omega$ is the pullback of the 1-form $dz$ on $\TT$. Alternately, an STS is a translation surface where the Euclidean polygons are taken to be axis-parallel unit squares. Note that fixing the embedding of the squares in this way up to translation gives a well defined notion of top, bottom, left and right for each of the four sides of a square. Square-tiled surfaces also admit a third, combinatorial description via pairs of permutations that describe the horizontal and vertical gluings. Given a square-tiled surface with $n$ squares, we first label the squares $1, \dots, n$ and then define two permutations $\sigma, \tau \in S_n$ as follows: square $i$ is to the right of square $j$ if and only if $\sigma(j) = i$ and square $i$ is to the top of square $j$ if and only if $\tau(j) = i$. On the other hand, given two permutations $\sigma, \tau \in S_n$, one can construct a labelled square-tiled surface, denoted $S(\sigma, \tau)$, with $n$ squares (possibly with multiple connected components). We will denote $\STS_n$ as the collection of square-tiled surfaces (not necessarily connected) with $n$ squares. Note that since STSs are built using axis parallel squares, they always decompose into cylinders in the horizontal direction. See Figure \ref{fig:introSTS} for an example.

The study of translation surfaces is motivated by their use in related areas of study such as Teichm\"{u}ller theory and polygonal billiards. Translation surfaces admit an $\SL_2(\RR)$ action and the orbits of translation surfaces under the action of the diagonal subgroup give rise to Teichm\"{u}ller geodesics in Teichm\"{u}ller space. They appear in the study of polygonal billiards via a classical construction due to Katok-Zemlyakov \cite{KatZem} of ``unfolding" a billiard trajectory which allows one to relate billiard trajectories to geodesics on a translation surface. 

Square-tiled surfaces themselves also have been actively studied and used. They are concrete, hands-on examples of translation surfaces with very nice properties. For instance, the stabilizer of an STS under the $\SL_2(\RR)$ action is a lattice. Translation surfaces with lattice stablizers are called \textbf{Veech surfaces} or \textbf{lattice surfaces} and are known to exhibit rich geometric and dynamical properties. For instance, they exhibit ``optimal dynamics" -- in any direction, infinite geodesic trajectories are either periodic or dense. Moreover, since STSs are made up of squares, they can also be seen as lattice points in the moduli space of translation surfaces and can be used for volume computations \cite{Zor2, EsOk, DelGuoZogZor}. 
Moreover, their square structure also opens up connections to analytic number theory and modular forms \cite{LelRoy,LemShresThor,  ShresCounting}. 

In this note we use the combinatorial description of STSs to study the topological and geometric statistics of random STSs. In \cite{ShresRandom, Lech}, leveraging the description of STSs as pairs of permutations, $S_n \times S_n$ with the uniform distribution was proposed as a random model for STSs in $\STS_n$ and the asymptotic distribution of the genus was computed along with asymptotic geometric information. We refer to this earlier model as the \textbf{Standard Model}. In this paper, we modify the standard random model studied in \cite{ShresRandom, Lech} to define a new model where the first permutation sampled is restricted to a fixed conjugacy class. The initial motivation (which we will expand on shortly) for looking at such a modification is to restrict the combinatorial configuration of horizontal cylinders in the resulting square-tiled surface. To state the model precisely, we fix $\alpha \in [0, 1]$ and let $\mu_n$ be an integer partition of $n$ chosen at random (uniformly) among all partitions with at most $n^\alpha$ parts. Let $\calK_{\mu_n}$ be the conjugacy class of $S_n$ consisting of permutations having the cycle type $\mu_n$. Then $\calK_{\mu_n} \times S_n$ with the uniform distribution gives a way to pick a random pair of permutations which can be used to build a square-tiled surface with $n$ squares.  We call this model the \textbf{Horizontally Restricted (HR) Model}. Denote by $x \sim \Unif(X)$ when $x$ is chosen uniformly at random from a finite set $X$. For $\alpha < \frac{1}{2}$, we consider topological and geometric properties of $S(\sigma, \tau)$ where $(\sigma, \tau) \sim \Unif(\calK_{\mu_n} \times S_n)$ as $n \rightarrow \infty$. In particular, we will see that these asymptotic topological and geometric properties will not depend on the specific sequence of partitions $\mu_n$ as long as each partition has at most $n^\alpha$ parts. 

\subsection{Motivation for the Horizontally Restricted (HR) Model of STSs}

In order to motivate the Horizontally Restricted (HR) Model with general parameter $\alpha$, we first consider the case when $\alpha = 0$. In this case $\calK_{\mu_n}$ is the conjugacy class of $n$-cycles and $\calK_{\mu_n} \times S_n$ with the uniform distribution is a random model for square-tiled surfaces with a single maximal horizontal cylinder of length $n$ and height 1 (i.e. the sole horizontal cylinder is comprised of a single band of squares). All such STSs are instances of \textbf{horizontally one-cylinder surfaces} i.e. the surface contains only one maximal horizontal cylinder. Note that a horizontally one-cylinder STS does not need to have a single band of squares  (see Figure~\ref{fig:oneCylSurfaces} for examples of ones with larger height). We emphasize, however, that $\calK_{\mu_n} \times S_n$ also acts as a good model for \emph{all} horizontally one-cylinder surfaces since the proportion of STSs whose horizontal cylinders have multiple bands of squares is asymptotically zero (see Theorem~\ref{thm:singleBanded} below in Section~\ref{sec:mainResults} that lists the main results of this paper). Hence, in general, the HR model with parameter $\alpha$ (i.e. restricting the $\sigma$ permutation to one with at most $n^\alpha$ cycles) serves as a model for STSs with at most $n^\alpha$ maximal horizontal cylinders. 

Studying STSs by restricting their horizontal cylinder configurations is not new and the importance of studying STSs in this way was highlighted in earlier work by Delecroix-Goujard-Zograf-Zorich \cite{DelGuoZogZor}. There the asymptotic (as genus grows) absolute and relative contribution of horizontally one-cylinder STSs to the volume of the ambient strata is computed using the fact that square-tiled surfaces of a fixed combinatorial configuration of horizontal cylinders equidistribute in the ambient strata. Moreover, in \cite{Zmi2} and \cite{ShresCounting}, STSs in genus two have been enumerated by their horizontal cylinder configurations (more precisely, their \emph{cylinder diagrams}). 

In addition to Theorem~\ref{thm:singleBanded}, previous study of random STSs by Delecroix-Groujard-Zograf-Zorich in  \cite{DelGuoZogZor2} provide more motivation as to why the HR model is a good model for STSs with at most $n^\alpha$ cylinders. In their work, the authors consider square-tiled surfaces that have a \textbf{half-translation} structure (where edge identifications are either via translations or translations plus a $\pi$-rotation.) Moreover, for their random model, they consider $\STS_g(N)$ to be the finite set of such STSs with genus $g$ and up to $N$ squares. Considering the uniform distribution on $\STS_g(N)$, they then let $N$ tend to infinity and define a random STS of genus $g$. 
With this model they show (among various other results) that given a sequence of integers $k_g = o(\log g)$, the probability that each maximal horizontal cylinder of a random $k_g$-cylinder STS of genus $g$ is composed of a single band of squares tends to 1 as $g \rightarrow \infty$. In other words, the asymptotic probability of maximal horizontal cylinders being composed of single bands of squares tends to 1 when the number of cylinders grows slow enough. Additionally, if one does not condition on the number of cylinders, they prove that the probability that each maximal horizontal cylinder of a random STS of genus $g$ has a single band of squares tends to $\frac{\sqrt{2}}{2}$ as $g \rightarrow \infty$.

\begin{figure}
\centering
\begin{subfigure}{0.4\textwidth}
\centering
\begin{tikzpicture}
    [sq/.style=
  {shape=regular polygon, regular polygon sides=4, draw, minimum width=1.414cm, inner sep = 0}, rect/.style={shape=regular polygon, regular polygon sides=4, draw, minimum width=1pt}]
\node[sq] at (0,0){$1$};
\node[sq] at (1,0){$2$};
\node[sq] at (2,0){$3$};
\node[sq] at (3,0){$4$};
\node[sq] at (4,0){$5$};
\node[sq] at (5,0){$6$};


\node[] at (-0.65,0){$a$};
\node[] at (5.65, 0){$a$};

\node[] at (0,0.7){$b$};
\node[] at (0, -0.7){$b$};

\node[] at (1,0.7){$c$};
\node[] at (2, -0.7){$c$};

\node[] at (2,0.7){$d$};
\node[] at (3, -0.7){$d$};

\node[] at (3,0.7){$e$};
\node[] at (1, -0.7){$e$};

\node[] at (4,0.7){$f$};
\node[] at (5, -0.7){$f$};

\node[] at (5,0.7){$g$};
\node[] at (4, -0.7){$g$};

\end{tikzpicture}
\caption{}
\end{subfigure}
\begin{subfigure}{0.25\textwidth}
\centering
\begin{tikzpicture}
    [sq/.style=
  {shape=regular polygon, regular polygon sides=4, draw, minimum width=1.414cm, inner sep = 0}, rect/.style={shape=regular polygon, regular polygon sides=4, draw, minimum width=1pt}]
\node[sq] at (0,0){$1$};
\node[sq] at (1,0){$2$};
\node[sq] at (2,0){$3$};
\node[sq] at (0,1){$4$};
\node[sq] at (1,1){$5$};
\node[sq] at (2,1){$6$};


\node[] at (-0.65,0){$a$};
\node[] at (2.65, 0){$a$};

\node[] at (-0.65,1){$b$};
\node[] at (2.65, 1){$b$};

\node[] at (0,1.7){$c$};
\node[] at (1, -0.7){$c$};

\node[] at (1,1.7){$d$};
\node[] at (0, -0.7){$d$};

\node[] at (2,1.7){$e$};
\node[] at (2, -0.7){$e$};






\end{tikzpicture}
\caption{}
\end{subfigure}
\begin{subfigure}{0.2\textwidth}
\centering
\begin{tikzpicture}
    [sq/.style=
  {shape=regular polygon, regular polygon sides=4, draw, minimum width=1.414cm, inner sep = 0}, rect/.style={shape=regular polygon, regular polygon sides=4, draw, minimum width=1pt}]
\node[sq] at (0,0){$1$};
\node[sq] at (1,0){$2$};
\node[sq] at (0,1){$3$};
\node[sq] at (1,1){$4$};
\node[sq] at (0,2){$5$};
\node[sq] at (1,2){$6$};

\node[] at (-0.65,0){$a$};
\node[] at (1.65, 0){$a$};

\node[] at (-0.65,1){$b$};
\node[] at (1.65, 1){$b$};

\node[] at (-0.65,2){$c$};
\node[] at (1.65, 2){$c$};

\node[] at (0,2.7){$d$};
\node[] at (1, -0.7){$d$};

\node[] at (1,2.7){$e$};
\node[] at (0, -0.7){$e$};

\end{tikzpicture}
\caption{}
\end{subfigure}
\begin{subfigure}{0.1\textwidth}
\centering
\begin{tikzpicture}
    [sq/.style=
  {shape=regular polygon, regular polygon sides=4, draw, minimum width=1.414cm, inner sep = 0}, rect/.style={shape=regular polygon, regular polygon sides=4, draw, minimum width=1pt}]
\node[sq] at (0,0){$1$};
\node[sq] at (0,1){$2$};
\node[sq] at (0,2){$3$};
\node[sq] at (0,3){$4$};
\node[sq] at (0,4){$5$};
\node[sq] at (0,5){$6$};

\node[] at (-0.65,0){$a$};
\node[] at (0.65, 0){$a$};

\node[] at (-0.65,1){$b$};
\node[] at (0.65, 1){$b$};

\node[] at (-0.65,2){$c$};
\node[] at (0.65, 2){$c$};

\node[] at (-0.65,3){$d$};
\node[] at (0.65, 3){$d$};

\node[] at (-0.65,4){$e$};
\node[] at (0.65, 4){$e$};

\node[] at (-0.65,5){$f$};
\node[] at (0.65, 5){$f$};

\node[] at (0,-0.7){$g$};
\node[] at (0, 5.7){$g$};

\end{tikzpicture}
\caption{}
\end{subfigure}
\caption{Examples of horizontally one-cylinder surfaces. Note that in the surface (A) the maximal horizontal cylinder is comprised of a single band of squares (i.e. the height of the cylinder is 1). The maximal horizontal cylinders of surfaces (B), (C), and (D) are comprised of 2, 3 and 6 bands of squares respectively.}
\label{fig:oneCylSurfaces}
\end{figure}

\subsection{Main Results}\label{sec:mainResults}

A square-tiled surface is called \textbf{single-banded} if all of its maximal horizontal cylinders are of height 1 (i.e. are comprised of a single band of squares). Our first main result concerns the asymptotic probability of such surfaces in the standard model. 

\begin{restatable}[Asymptotic Probability of Single-Banded STSs]{theorem}{singlebanded} \label{thm:singleBanded}
Let $(\sigma, \tau) \sim \Unif(S_n \times S_n)$. Then the probability that $S(\sigma, \tau)$ is single-banded is asymptotically 1. More precisely,
    $$\Pr[S(\sigma, \tau) \text{ is single-banded}] = 1 - O\left(\frac{1}{n}\right)$$   
\end{restatable}

Our second main result concerns the moments of the number of vertices in a random STS from the HR model. Since the number of vertices is intimately related to the genus (by the Euler characteristic formula), the moments of the number of vertices is crucial in obtaining the distribution of the genus.

\begin{restatable}[Moments of Number of Vertices]{theorem}{moments} \label{thm:moments} Fix $\alpha \in [0, \frac{1}{2})$ and let $\mu_n \vdash n$ be a partition of $n$ with at most $n^\alpha$ parts. Let $V_{\mu_n}$ be the number of vertices of $S(\sigma, \tau)$ for $(\sigma,\tau) \sim \Unif(\calK_{\mu_n} \times S_n)$. Let $p(x)$ be a polynomial of degree $l$ such that $p(x)$ is non-decreasing and $p(x) \geq 0$ for $x \geq 0$. Then, for any $\epsilon > 0$ such that $\alpha + \epsilon < \frac{1}{2}$,
$$\EE[p(V_{\mu_n})] = \EE[p(C_{n})] + O\left(\frac{\log^l(n)}{n^{1-2\alpha-2\epsilon}}\right) = p(\log(n)) + O(\log(n)^{l-1}) + O\left(\frac{\log^l(n)}{n^{1-2\alpha-2\epsilon}}\right)  $$
where $C_n$ is the number of cycles of a uniformly random permutation chosen from the alternating group $A_n$. 
\end{restatable}

Note that we can use the Euler characteristic formula to obtain the genus of a surface $S \in \STS_n$ as given by
$$g = \frac{n}{2} - \frac{V}{2} + 1$$ 
where $V$ is the number of vertices. So, Theorem \ref{thm:moments} also gives the mean and variance of the genus of a random STS in the HR model as a Corollary:

\begin{restatable}[Expected Genus]{corollary}{expectedgenus}\label{cor:expectedgenus} Fix $\alpha \in [0, \frac{1}{2})$ and let $\mu_n \vdash n$ a partition of $n$ with at most $n^\alpha$ parts. Let $G_{\mu_n} = \frac{n}{2} - \frac{V_{\mu_n}}{2} +1 $ be the genus of $S(\sigma, \tau)$ for $(\sigma, \tau) \sim \Unif(\calK_{\mu_n} \times S_n)$. Then the expectation and variance of $G_{\mu_n}$ is given by,
$$\EE[G_{\mu_n}] = \frac{n}{2} - \frac{\log n}{2} - \frac{\gamma}{2} +1 +o(1)\,\, \text{ and } \,\,\Var[G_{\mu_n}] = \frac{\log n}{4} +\frac{\gamma}{4} - \frac{\pi^2}{24} +o(1)$$ 
where $\gamma \approx 0.577$ is the Euler-Mascheroni constant.
\end{restatable}

In many cases, knowing all the moments of a random variable gives sufficient information to recover the distribution. Hence, we are also able to obtain the asymptotic distribution of the genus of a random STS in the HR Model.

\begin{restatable}[Genus Distribution]{theorem}{genusdistribution}\label{thm:genusdist} Let $\alpha \in [0, \frac{1}{2}) $ and $\mu_n \vdash n$ with at most $n^\alpha$ parts. Let $G_{\mu_n} = \frac{n}{2} - \frac{V_{\mu_n}}{2} +1$ be the genus of $S(\sigma, \tau)$ for $(\sigma, \tau) \sim \Unif(\calK_{\mu_n} \times S_n)$. Let $a > 0$ fixed and $K_n$ be the number of cycles of $\pi \sim \Unif(S_n)$. Then, uniformly, for $\ell$ satisfying $n-l$ being even and $\frac{\ell - \EE[K_n] }{\sqrt{\Var[K_n]}} \in [-a,a]$ , 

$$\Pr\left(G_{\mu_n} = \frac{n}{2} - \frac{\ell}{2} + 1 \right) =   \frac{(2 + O(\log(n)^{-1/2} ))e^{-\frac{1}{2} \cdot \frac{(\ell - \EE[K_n])^2}{ \Var[K_n]}}}{\sqrt{2 \pi \Var[K_n]}} $$
where $\EE[K_n] = \log(n) + \gamma + o(1)$, $\Var[K_n] = \log(n) + \gamma -\frac{\pi^2}{6}+ o(1) $ and $\gamma \approx 0.577$ is the Euler-Mascheroni constant.
\end{restatable}

Translation surfaces of genus $g$ can be partitioned further by grouping them by the combinatorial data of their cone points. In particular let $\calH(\alpha_1, \dots, \alpha_s)$ be the set of translation surfaces with $s$ cone points of angles $2\pi(\alpha_1+1), \dots, 2 \pi(\alpha_s+1)$. Such a set is called a \textbf{stratum}. In our next main result we find a sufficient criterion for when a random model for square-tiled surfaces ensures that the most likely stratum is the stratum with a single cone point with large angle. 

To make this more precise, consider a random model for choosing a labelled square-tiled surface with $n$ squares. By considering the horizontal and vertical gluings, from each such random surface we get two permutations $\sigma, \tau \in S_n$. Furthermore, considering the commutator $[\sigma,\tau]$, we obtain a random permutation in the alternating group $A_n$. Hence, a random model for choosing a labelled STS with $n$ squares induces a distribution on $A_n$. Let $P_n$ be this distribution and let $U_n$ be the uniform distribution on $A_n$. Then, we have the following result:

\begin{restatable}[Sufficient Criterion for Most Likely Stratum]{theorem}{sufficientcriterionmostlikelystratum}\label{thm:sufficientcriterion}
Let $P_n$ be as above. Assume that $P_n$ is constant on conjugacy classes of $S_n$ and for real number $s > 1$ satisfies the following equation:
\begin{equation}\label{eqn:l2bound}|A_n|\sum_{\sigma \in A_n}|P_n(\sigma) - U_n(\sigma)|^2 = O\left(\frac{1}{n^s}\right)\end{equation}
Then there exists some $N \in \mathbb{N}$ such that for all  $n > N$, the most probable stratum is $\mathcal{H}(n-1)$ if $n$ is odd, and $\mathcal{H}(n-2)$ if $n$ is even. 
\end{restatable}
Theorem~\ref{thm:sufficientcriterion} can be applied in two cases, giving us a useful corollary in our case. In the Standard Model, we use the uniform distribution on $S_n \times S_n$ to pick a random pair of permutations $(\sigma, \tau)$ which induces (via the commutator $[\sigma, \tau]$) distribution $P_n$ on $A_n$. Likewise in the HR model with parameter $\alpha \in [0, \frac{1}{4}]$, we use a modified distribution to pick a pair of permutations which in turn induces a distribution $P_n$ on $A_n$ via the commutator. Both these induced distributions satisfy the hypothesis of Theorem~\ref{thm:sufficientcriterion}, yielding the following corollary:
\begin{restatable}[Most Likely Stratum]{corollary}{mostlikelystratum}\label{cor:mostlikelyStrataspecific} Let $\alpha \in [0, \frac{1}{4})$ and  $\mu\vdash n$ with at most $n^\alpha$ parts. Let $(\sigma, \tau) \sim \Unif(S_n \times S_n)$ or $\Unif(\calK_{\mu_n} \times S_n)$. There exists $N \in \NN$ such that for all $n > N$, $\Pr(S(\sigma, \tau) \in \calH(\alpha))$ is maximized when $\calH(\alpha) = \calH(n-1)$ (if $n$ is odd) and $\calH(\alpha) = \calH(n-2)$ (if $n$ is even.)
In other words, the most likely stratum in both the standard model and the $\alpha$-based HR model is either $\calH(n-1)$ if $n$ is odd or $\calH(n-2)$ if $n$ is even. 
\end{restatable}

We also deduce some geometric information about random STSs from our model. One way to study translation surfaces is via \textbf{saddle connections} which are straight lines (i.e. local geodesics for the flat metric) on the surface between cone point(s), without any cone points in the interior. Associated to a saddle connection is a vector in $\RR^2$ called its \textbf{holonomy vector} which records its Euclidean $x$ and $y$ displacement. We denote the set of holonomy vectors of a surface $S$ as $\Hol(S) \subseteq \RR^2$. Given a square-tiled surface, $S$, built out of unit squares, it is easily seen that $\Hol(S) \subseteq \ZZ^2$ since the cone points can only appear at the vertices. Marking a point, we note that the holonomy vectors of the standard square-torus is the set $\Hol(\TT)=\{(p,q) \in \ZZ^2|\gcd(p,q) = 1\}$. A surface $S$ is called a \textbf{holonomy torus} if $\Hol(S) = \Hol(\TT)$ and a \textbf{visibility STS} if $\Hol(S) \supseteq \Hol(\TT)$. Our next result concerns the probability of a random STS being one of these two. 

\begin{restatable}[Holonomy Theorem]{theorem}{holonomy} \label{thm:holonomy} 
   Let $\alpha \in [0,\frac{1}{2})$ be fixed and let $\mu_n \vdash n$ be a partition of $n$ with at most $n^\alpha$ parts. Let $(\sigma, \tau) \sim \Unif(\calK_{\mu_n} \times S_n)$. Then,

    \begin{enumerate} 
    \item The probability that $S(\sigma, \tau)$ is a holonomy torus is asymptotically  $\frac{1}{e}$. More precisely, for any $\epsilon > 0$ such that $\alpha + \epsilon < \frac{1}{2}$,
     $$\left|\Pr(S(\sigma, \tau) \text{ is a holonomy torus}) - \frac{1}{e}\right| = O\left(\frac{1}{n^{1-2\alpha-2\epsilon}}\right)$$
    
    \item The probability that $S(\sigma, \tau)$ is a visibility STS is asymptotically 1. More precisely, 
    $$\Pr(S(\sigma, \tau) \text{ is a visibility STS}) = 1 - O\left(\frac{n^2}{2^{n/2}}\right)$$   
    \end{enumerate}
\end{restatable}

\subsection{Standard vs. HR Models}
The results Theorem \ref{thm:moments}, Corollary \ref{cor:expectedgenus}, Theorem \ref{thm:genusdist} and Theorem \ref{thm:holonomy} are HR-model counterparts to Theorem 2.4, Corollary 2.5, Theorem 2.6, and the Holonomy Theorem proven in \cite{ShresRandom} for the standard model. While some of the techniques between the two models are similar, in this paper we've had to obtain newer versions of previous lemmas to carefully adapt for the new model. For instances, see Lemma~\ref{lem:NumberOfProductPairs} and Lemma~\ref{lem:largedeviations}. Likewise, we also use stronger estimates on character values. For instance, see Theorem~\ref{thm:larsenshalevbound}. 

Theorem~\ref{thm:sufficientcriterion} can be more generally applied to other potential random models for STSs. We apply it to obtain Corollary~\ref{cor:mostlikelyStrataspecific} which is a result that concerns both the Standard and HR models. These latter results do not have counterparts in \cite{ShresRandom}. Additionally, although Theorem~\ref{thm:singleBanded} concerns the Standard model, it is a new result in this paper.


\subsection{Broader context}

The topology of random surfaces has had a rich history. Gamburd-Makover \cite{GamMak} and Pippenger-Schleich \cite{PipSchleich} studied surfaces built out of $n$ hyperbolic ideal triangles and computed the expected genus as  $n/4-\log(n)/2 + O(1)$. Similarly, Fleming-Pippenger \cite{FlemPip} considered surfaces built out of $N/k$ $k$-gons with total $N$ sides and proved a theorem analogous to our Theorem \ref{thm:moments} on moments of vertices. Chmutov-Pittel \cite{ChmuPit} considered surfaces built out of $n$ polygons with total number of sides $N$ and proved a theorem analogous to our Theorem \ref{thm:genusdist}, in particular, showing that the expected genus is asymptotic to $(N/2 -n-\log N)/2$. Combined works of Harer-Zagier\cite{HarZag}, Linial-Nowik \cite{LinNow} and Chmutov-Pittel \cite{ChmuPitt2}, consider a random surface obtained by gluing sides of an $n$-gon and compute the asymptotic distribution of the genus, proving that the expected genus is $n/2 - \Theta(\ln n)$.

The geometric statistics of random surfaces has also been well studied. Brooks-Makover \cite{BrooksMak}, Guth-Parlier-Young \cite{GuthParYoung}, Mirzakhani-Petri \cite{MirPet}, Petri \cite{Pet1, Pet2}, Petri-Thale \cite{PetTha}, Budzinski-Curien-Petri \cite{BudCurPet2}, Liu-Petri \cite{LiuPet} and others have studied Cheeger constants, systoles, length spectrum, pants length, largest embedded balls, diameter etc. of random (not necessarily translation) surfaces. 

Random translation surfaces have also gained some attention recently. Any connected component of a stratum of translation surfaces can be equipped with a measure called the Masur-Veech measure. With respect to this measure Masur-Rafi-Randecker \cite{MasRafRan} studied the expected covering radius of a transation surface in a given stratum and showed that it is bounded above by $20 \sqrt{\log (g)/g}$ when the genus $g$ is large. 

Since square-tiled surfaces are branched covers of the torus, the study of random STSs also falls in the general framework of the study of random covers of manifolds. Recently, Magee-Puder \cite{MagPud} and Magee-Naud-Puder \cite{MagNaudPud} studied random covers of compact hyperbolic surfaces while Baker-Petri \cite{BakPet} studied random covers of torus knot complements. 

Likewise, saddle connections and holonomy vectors are also well studied and are important objects in the theory of translation surfaces. For instance, it has been shown that any stratum of translation surfaces itself is an orbifold with local coordinates given by holonomy vectors of saddle connections that form a basis for relative homology. See, for instance, Section 2.3 of \cite{ForMat} for more details. Moreover, the set of holonomy vectors of a translation surface has been used by Smilie-Weiss \cite{SmiWeiss} and Forester-Tang-Tao \cite{ForTangTao} to characterize Veech surfaces. 

Given their importance, naturally there has been wide interest in the statistics of saddle connections on translation surfaces. Combined works of Masur \cite{MasLowerBound, MasGrowthRate}, Eskin-Masur \cite{EskMas} and Nevo-R\"uhr-Weiss \cite{NevoRuhrWeiss} have established (and refined) that in any stratum, almost every translation surface has an exact quadratic asymptotic growth for the number of saddle connections up to length $R$ with the constant solely depending on the stratum. More recent work on counting saddle connections and holonomy vectors include Athreya-Fairchild-Masur \cite{AthFairMas} who prove quadratic growth rate for pairs of saddle connections up to length $R$ with bounded cross product of their holonomy vectors. 

\subsection{Struture of the Paper}
In Section \ref{sec:background} we give the necessary background for the proofs of our results. In particular, we give a brief introduction to translation and square-tiled surfaces followed by a concrete definition of the random models we consider. Since these models are based on permutations, we also introduce key statistics of random permutations that will be useful in our computations. The representation theory of the symmetric and alternating groups play a vital role in our proofs so we also state relevant foundational definitions and past results. As permutations of $n$ are intimately tied to integer partitions of $n$, we also briefly introduce some combinatorial tools pertaining to these. 

In Section~\ref{sec:singleBanded} we prove that in the Standard model, the asymptotic probability that a random STS is single-banded is 1 (Theorem \ref{thm:singleBanded}). In Section \ref{sec:convergenceofdist} we consider the distribution $P_n$ on $A_n$ that is induced by our HR model. In particular we prove that this distribution (when $\alpha < 1/2$) converges to the uniform distribution on $A_n$. Section \ref{sec:moments} is devoted to proving Theorem \ref{thm:moments}. In Section \ref{sec:topologicalstatistics} we prove Theorems \ref{thm:genusdist}, \ref{cor:expectedgenus} and Corollary \ref{cor:mostlikelyStrataspecific}. In this section we also show that the HR model produces connected square-tiled surfaces asymptotically almost surely. Finally Section \ref{sec:holonomy} is devoted to proving Theorem \ref{thm:holonomy}.

\subsection{Acknowledgements} We would like to express our gratitude to Jane Wang for helpful conversations regarding the project. Some of the work in this paper was done as part of authors Fitzhugh, Schondorf, Vander Ploeg Fallon, Zeng's Integrated Comprehensive Exercise (i.e. senior capstone project) for Mathematics at Carleton College, MN, USA. Hence, we also acknowledge the Carleton College Mathematics Department for the opportunity to work on this project. 

\section{Background}\label{sec:background}
\subsection{Translation surfaces and strata}

We start with brief background on translation surfaces. For a more complete reference, please see \cite{Zor1} or \cite{Mas}.

A \textbf{translation surface} is a pair $(X, \omega)$ where $X$ is a Riemann surface and $\omega$ is a holomorphic 1-form. Alternately (and equivalently), a translation surface is a finite collection of Euclidean polygons, embedded in the plane with the embedding fixed only up to translations, and sides of the polygons identified in parallel opposite pairs up to equivalence by cut, translate and paste operations. See Figures \ref{fig:sc} for examples. 

\begin{figure}[h!]
\begin{minipage}{0.45\textwidth}
    \centering
    \resizebox{5cm}{!}{
\begin{tikzpicture}[hex/.style=
 {shape=regular polygon, regular polygon sides=6, draw, minimum width=2cm}]
 
 \node[hex] at (0:0cm) {};
 \node[hex] at (-30:1.72cm) {};


 \draw[] (2*30:1cm) circle [radius=2pt];
 \draw[] (6*30:1cm) circle [radius=2pt];
  \draw[] (10*30:1cm) circle [radius=2pt];
 \draw[shift=({1.5,-0.86cm})] (2*30:1cm) circle [radius=2pt];
 \draw[shift=({1.5,-0.86cm})] (10*30:1cm) circle [radius=2pt];

  \draw[fill=black] (0*30:1cm) circle [radius=2pt];
  \draw[fill=black] (4*30:1cm) circle [radius=2pt];
  \draw[fill=black] (8*30:1cm) circle [radius=2pt];

  \draw[fill=black, shift=({1.5,-0.86cm})] (0*30:1cm) circle [radius=2pt];
  \draw[fill=black, shift=({1.5,-0.86cm})] (8*30:1cm) circle [radius=2pt];
 
 
 	\draw[thin] (0,0.86) -- ++(-0.15,0.15);
  \draw[thin] (0,0.86) -- ++(-0.15,-0.15);
  
  \draw[thin] (1.6,-1.72) -- ++(-0.15,0.15);
  \draw[thin] (1.6,-1.72) -- ++(-0.15,-0.15);

 
 \draw[thin] (-0.1,-0.86) -- ++(-0.15,0.15);
  \draw[thin] (-0.1,-0.86) -- ++(-0.15,-0.15);
 \draw[thin] (0.1,-0.86) -- ++(-0.15,0.15);
  \draw[thin] (0.1,-0.86) -- ++(-0.15,-0.15);
  
  	 \draw[thin] (1.4,0) -- ++(-0.15,0.15);
  \draw[thin] (1.4,0) -- ++(-0.15,-0.15);
   \draw[thin] (1.6,0) -- ++(-0.15,0.15);
  \draw[thin] (1.6,0) -- ++(-0.15,-0.15);
  

 \draw[thin] (1*30:0.71cm) -- (1*30:1.01cm);
\draw[thin][shift=({0,-1.72cm})] (1*30:0.71cm) -- (1*30:1.01cm);


\node[rotate=-60] at (7*30:0.86cm){$\vee$}; 
\node[shift=({1.5,-0.86cm}),rotate=-60 ] at (1*30:0.86cm){$\vee$}; 


\node[rotate=150] at (5*30:0.86cm){$\sim$};
 \node[shift=({1.5,-0.86cm}),rotate=150 ] at (11*30:0.86cm){$\sim$}; 

\end{tikzpicture}
}
\end{minipage}
\begin{minipage}{0.45\textwidth}
\centering
\resizebox{5cm}{!}{
\begin{tikzpicture}[oct/.style=
 {shape=regular polygon, regular polygon sides=8, draw, minimum width=2cm},decoration={
    markings,
    mark=at position 0.5 with {\arrow{>}}}]
 
 \node[oct] at (-6*22.5:1.85cm) {};

  \begin{scope}[shift={(-6*22.5:1.85cm)}]

  \draw (4*22.5:0.8cm) -- (4*22.5:1.05cm);
  \draw (12*22.5:0.8cm) -- (12*22.5:1.05cm);
  
  \draw [shift={(0,0.025)}] (8*22.5:0.8cm) -- (8*22.5:1.05cm);
  \draw [shift={(0,-0.025)}] (8*22.5:0.8cm) -- (8*22.5:1.05cm);
  
  \draw [shift={(0,0.025)}] (0:0.8cm) -- (0:1.05cm);
  \draw [shift={(0,-0.025)}] (0:0.8cm) -- (0:1.05cm);

  \draw [shift={(-0.05,0.05)}](2*22.5:0.8cm) -- (2*22.5:1.05cm);
 \draw (2*22.5:0.8cm) -- (2*22.5:1.05cm);
 \draw [shift={(0.05,-0.05)}](2*22.5:0.8cm) -- (2*22.5:1.05cm);

  \draw [shift={(-0.05,0.05)}](-6*22.5:0.8cm) -- (-6*22.5:1.05cm);
 \draw (-6*22.5:0.8cm) -- (-6*22.5:1.05cm);
 \draw [shift={(0.05,-0.05)}](-6*22.5:0.8cm) -- (-6*22.5:1.05cm);

 \draw [shift={(0.025,0.025)}](6*22.5:0.8cm) -- (6*22.5:1.05cm);
 \draw [shift={(-0.025,-0.025)}](6*22.5:0.8cm) -- (6*22.5:1.05cm);
 
 \draw [shift={(0.075,0.075)}](6*22.5:0.8cm) -- (6*22.5:1.05cm);
 \draw [shift={(-0.075,-0.075)}](6*22.5:0.8cm) -- (6*22.5:1.05cm);

 \draw [shift={(0.025,0.025)}](-2*22.5:0.8cm) -- (-2*22.5:1.05cm);
 \draw [shift={(-0.025,-0.025)}](-2*22.5:0.8cm) -- (-2*22.5:1.05cm);
 
 \draw [shift={(0.075,0.075)}](-2*22.5:0.8cm) -- (-2*22.5:1.05cm);
 \draw [shift={(-0.075,-0.075)}](-2*22.5:0.8cm) -- (-2*22.5:1.05cm);

  \foreach \x in {1,3,5,7,9,11,13,15}{
  \draw[fill=black] (\x*22.5:1cm) circle [radius=1.5pt];}

  \draw[dotted,postaction={decorate}] (11*22.5:1cm) -- (7*22.5:1cm);

   \draw[dashed,postaction={decorate}] (11*22.5:1cm) -- (4*22.5:1cm);
     \draw[dashed,postaction={decorate}] (12*22.5:1cm) -- (3*22.5:1cm);
  \end{scope}

\end{tikzpicture}
}
\end{minipage}

 \caption{Two translation surfaces. On the left is a surface formed by two hexagons on the right is one formed by identification of opposite sides of an octagon. Edge identifications are indicated by matching edge decorations. Cone points are represented by decorated vertices -- matching decorations mean that the vertices are identified under the edge gluings. The surfaces live in $\calH(1,1)$ and $\calH(2)$ respectively. The octagon surface also has two saddle connections on it (represented by dotted and dashed lines). Taking the Octagon sidelength to be 1, the holonomy vector corresponding to the dotted saddle connection is $\left(-\frac{1}{\sqrt{2}}, 1+\frac{1}{\sqrt{2}}\right)$ and the holonomy vector corresponding to the dashed saddle connection is $\left(1, 2+2\sqrt{2}\right)$. }
 \label{fig:sc}

\end{figure}

Geometrically, since translation surfaces are made up of Euclidean polygons, they are locally flat except for possibly at finitely many \textbf{cone points} or \textbf{singularities} (typically coming from vertices of the polygons) where the angle is $2\pi \ell$ for some integer $\ell > 1$. In the complex analytic point of view, a cone point of angle $2 \pi \ell$ corresponds to a zero of the 1-form of order $\ell -1$.

One way to organize translation surfaces is according to their cone point data specified by an unordered $s$-tuple $(\alpha_1, \dots, \alpha_s)$ describing $s$ cone points of angles $2 \pi (\alpha_i+1)$. In this way, the $\alpha_i$ record the angle excess at each cone point. Such a collection of translation surfaces sharing the same cone point data is called a \textbf{stratum} and denoted $\calH(\alpha)$ where $\alpha = (\alpha_1, \dots, \alpha_s)$. Since translation surfaces are flat everywhere away from the cone points, all the curvature is concentrated at the cone points so that a genus $g$ translation surface with cone point data $(\alpha_1, \dots, \alpha_s)$ satisfies a combinatorial Gauss-Bonnet relation: $\sum_{i=1}^s \alpha_i = 2g -2$. Hence, genus $g$ translation surfaces are partitioned into various strata and it is also known that any partition of $2g-2$ into positive integers determines a non-empty stratum. 

A \textbf{saddle connection} on a translation surface is any local geodesic curve (i.e. straight line) joining cone point(s) with no cone point in the interior. The \textbf{holonomy vector} of a saddle connection is a vector in $\RR^2$ that records the Euclidean $x$ and $y$ displacement of the saddle connection. 

For surfaces without cone points, we simply mark a point in order to study their saddle connections. For instance, since the square-torus $\TT = \CC/\ZZ[i]$ does not have a cone point we mark the origin and consider saddle connections from the origin to itself. We denote the set of all holonomy vectors associated to a translation surface $X$ as $\Hol(X)$. For instance, $\Hol(\TT)$ is the set of primitive vectors in $\ZZ^2$. If the translation surface $(X, \omega)$ has multiple connected components, then we will take $\Hol(X)$ to be the union of the holonomy vectors of each component. 

A \textbf{cylinder} on a translation surface is an isometric copy of a Euclidean cylinder $(\RR/ c\ZZ) \times (0, h)$ with length (of the core curve) $c \in \RR_> 0$ and height $h\in \RR_>0$ whose boundary is a union of saddle connections. Cylinders are maximal in that there are no cone points in the interior. These also play an important role in the theory of translation surfaces and studying their deformations can give information about the structure of the ambient stratum of translation surfaces. 

\subsection{Square-tiled surfaces} \label{sec:STS}

A \textbf{square-tiled surface} (STS) is a translation surface where the polygons comprising the surface are Euclidean unit squares embedded (up to translation) in the plane so that the sides are parallel to the axes. Note that since STSs are instances of translation surfaces, the sides are glued in parallel opposite pairs. Alternately, STSs are covers, possibly branched, of the standard square-torus, $\TT = \CC/\ZZ[i]$, where the branching is precisely over the origin. The number of squares used in the polygonal description corresponds to the degree of the cover. Let $\STS_n$ denote the collection of all square-tiled surfaces with $n$ squares. We do not require a surface in $\STS_n$ to be connected. 

Square-tiled surfaces can also be described combinatorially using pairs of permutations. For this, we start by labelling the squares of  a given $S \in \STS_n$ by $\{1, \dots, n\}$. To such a labelling we can associate a pair of permutations $\sigma, \tau \in S_n$, where $\sigma(i) = j$ if and only if square $j$ is glued to the right of square $i$ and $\tau(i) = j$ if and only if square $j$ is glued to the top of square $i$. See Figure \ref{fig:introSTS}. We will denote by $S(\sigma, \tau)$, the labelled square-tiled surface with permutation pair $(\sigma, \tau)$. If we change the labelling of the squares using a permutation $\pi$, the resulting permutation pair will be simultaneously conjugated by $\pi$. 

\begin{figure}
\begin{minipage}{0.33\textwidth}
\centering
\resizebox{5.5cm}{!}{
\begin{tikzpicture}
    [sq/.style=
  {shape=regular polygon, regular polygon sides=4, draw, minimum width=1.414cm}, rect/.style={shape=regular polygon, regular polygon sides=4, draw, minimum width=1pt}]
  
  

   \node[sq] at (2,1){};
  \node[sq] at (0,0){};
  
   \node[sq] at (1,0){};
   \node[sq] at (-1,0){};
   \node[sq] at (1,1){};
  
   \node[sq] at (-1,1){};
   \node[sq] at (-2,1){};

    \foreach \x in {3,5}{
   \draw[fill=black] (\x*45:0.707cm) circle [radius=2pt];
   \draw[shift=({-2,1}),fill=black] (\x*45:0.707cm) circle [radius=2pt];}
   \draw[shift=({-1,1}),fill=black] (1*45:0.707cm) circle [radius=2pt];
    \foreach \x in {1,7}{
    \draw (\x*45:0.607cm) rectangle (\x*45:0.807cm);
    \draw[shift=({2,1})] (\x*45:0.607cm) rectangle (\x*45:0.807cm);}
    \draw[shift=({1,1})](3*45:0.607cm) rectangle (3*45:0.807cm);

    \foreach \x in {3, 5}{
    \draw[shift=({-1,0})] (\x*45:0.707cm) circle [radius=2.5pt];
    }
    \foreach \x in {1, 7}{
    \draw[shift=({1,0})] (\x*45:0.707cm) circle [radius=2.5pt];
    }
    \draw[shift=({-1,1})] (3*45:0.707cm) circle [radius=2.5pt];
  \draw[shift=({1,1})] (1*45:0.707cm) circle [radius=2.5pt];

    \draw[thin] (0,0.35) -- ++(0,0.3);
    \draw[thin] (0,-0.35) -- ++(0,-0.3);

    \draw[thin] (-1.05,1.35) -- ++(0,0.3);
    \draw[thin] (-1.05,-0.35) -- ++(0,-0.3);
    \draw[thin] (-0.95,1.35) -- ++(0,0.3);
    \draw[thin] (-0.95,-0.35) -- ++(0,-0.3);
  
    \draw[thin] (1,1.35) -- ++(0,0.3);
    \draw[thin] (1,-0.35) -- ++(0,-0.3);
    \draw[thin] (1.05,1.35) -- ++(0,0.3);
    \draw[thin] (1.05,-0.35) -- ++(0,-0.3);
    \draw[thin] (0.95,1.35) -- ++(0,0.3);
    \draw[thin] (0.95,-0.35) -- ++(0,-0.3);

  \draw[thin] (-1.5,0.1) -- ++(-0.15,-0.15);
  \draw[thin] (-1.5,0.1) -- ++(0.15,-0.15);
  \draw[thin] (1.5,0.1) -- ++(-0.15,-0.15);
  \draw[thin] (1.5,0.1) -- ++(0.15,-0.15);


  \draw[thin] (-2.5,0.9) -- ++(-0.15,-0.15);
  \draw[thin] (-2.5,0.9) -- ++(0.15,-0.15);
  \draw[thin] (-2.5,1.1) -- ++(-0.15,-0.15);
  \draw[thin] (-2.5,1.1) -- ++(0.15,-0.15);
  
  \draw[thin] (-0.5,0.9) -- ++(-0.15,-0.15);
  \draw[thin] (-0.5,0.9) -- ++(0.15,-0.15);
  \draw[thin] (-0.5,1.1) -- ++(-0.15,-0.15);
  \draw[thin] (-0.5,1.1) -- ++(0.15,-0.15);

\draw[thin] (0.5,1.1) -- ++(-0.15,-0.15);
  \draw[thin] (0.5,1.1) -- ++(0.15,-0.15);
  \draw[thin] (0.5,1) -- ++(-0.15,-0.15);
  \draw[thin] (0.5,1) -- ++(0.15,-0.15);
  \draw[thin] (0.5,0.9) -- ++(-0.15,-0.15);
  \draw[thin] (0.5,0.9) -- ++(0.15,-0.15);

  \draw[thin] (2.5,1.1) -- ++(-0.15,-0.15);
  \draw[thin] (2.5,1.1) -- ++(0.15,-0.15);
  \draw[thin] (2.5,1) -- ++(-0.15,-0.15);
  \draw[thin] (2.5,1) -- ++(0.15,-0.15);
  \draw[thin] (2.5,0.9) -- ++(-0.15,-0.15);
  \draw[thin] (2.5,0.9) -- ++(0.15,-0.15);

\node[] at (-2, 1.5){$\vee$};
\node[] at (-2, 0.5){$\vee$};

\node[rotate=90] at (2, 1.5){$\sim$};
\node[rotate=90] at (2, 0.5){$\sim$};

\end{tikzpicture}
}
    
\end{minipage}
\begin{minipage}{0.33\textwidth}
\centering
\resizebox{4cm}{!}{
\begin{tikzpicture}
    [sq/.style=
  {shape=regular polygon, regular polygon sides=4, draw, minimum width=1.414cm}]
  

\node[sq] at (-1,0){};
   \node[] at (0,0){};
   \node[sq] at (1,0){};
    \node[sq] at (0,1){};
\node[sq] at (0,-1){};
   \foreach \x in {1,3,5,7}{
   \draw[fill=black] (\x*45:0.707cm) circle [radius=2pt];}
   
  \draw[thin] (0,1.35) -- ++(0,0.3);
  \draw[thin] (0,-1.35) -- ++(0,-0.3);
  
  \draw[thin] (-1.05,0.35) -- ++(0,0.3);
  \draw[thin] (-1.05,-0.35) -- ++(0,-0.3);
  \draw[thin] (-0.95,0.35) -- ++(0,0.3);
  \draw[thin] (-0.95,-0.35) -- ++(0,-0.3);
  
  \draw[thin] (1,0.35) -- ++(0,0.3);
  \draw[thin] (1,-0.35) -- ++(0,-0.3);
  \draw[thin] (1.05,0.35) -- ++(0,0.3);
  \draw[thin] (1.05,-0.35) -- ++(0,-0.3);
  \draw[thin] (0.95,0.35) -- ++(0,0.3);
  \draw[thin] (0.95,-0.35) -- ++(0,-0.3);
  
  \draw[thin] (-0.5,1.1) -- ++(-0.15,-0.15);
  \draw[thin] (-0.5,1.1) -- ++(0.15,-0.15);
  
  \draw[thin] (0.5,1.1) -- ++(-0.15,-0.15);
  \draw[thin] (0.5,1.1) -- ++(0.15,-0.15);
  
\draw[thin] (-0.5,-0.9) -- ++(-0.15,-0.15);
  \draw[thin] (-0.5,-0.9) -- ++(0.15,-0.15);
  \draw[thin] (-0.5,-1.1) -- ++(-0.15,-0.15);
  \draw[thin] (-0.5,-1.1) -- ++(0.15,-0.15);
  
  \draw[thin] (0.5,-0.9) -- ++(-0.15,-0.15);
  \draw[thin] (0.5,-0.9) -- ++(0.15,-0.15);
  \draw[thin] (0.5,-1.1) -- ++(-0.15,-0.15);
  \draw[thin] (0.5,-1.1) -- ++(0.15,-0.15);


\draw[thin] (-1.5,0.1) -- ++(-0.15,-0.15);
  \draw[thin] (-1.5,0.1) -- ++(0.15,-0.15);
  \draw[thin] (-1.5,0) -- ++(-0.15,-0.15);
  \draw[thin] (-1.5,0) -- ++(0.15,-0.15);
  \draw[thin] (-1.5,-0.1) -- ++(-0.15,-0.15);
  \draw[thin] (-1.5,-0.1) -- ++(0.15,-0.15);

  \draw[thin] (1.5,0.1) -- ++(-0.15,-0.15);
  \draw[thin] (1.5,0.1) -- ++(0.15,-0.15);
  \draw[thin] (1.5,0) -- ++(-0.15,-0.15);
  \draw[thin] (1.5,0) -- ++(0.15,-0.15);
  \draw[thin] (1.5,-0.1) -- ++(-0.15,-0.15);
  \draw[thin] (1.5,-0.1) -- ++(0.15,-0.15);

\end{tikzpicture}
}
\end{minipage}
\begin{minipage}{0.3\textwidth}
\centering
\resizebox{5cm}{!}{
\begin{tikzpicture}
    [sq/.style=
  {shape=regular polygon, regular polygon sides=4, draw, minimum width=1.414cm}]
  

   \node[sq] at (0,0){};
   \node[sq] at (1,0){};
    \node[sq] at (2,0){};
\node[sq] at (3,0){};

\node[sq] at (0,1){};
\node[sq] at (1,1){};
    \foreach \x in {3,5}{
    \draw[fill=black] (\x*45:0.707cm) circle [radius=2pt];
    \draw[fill=black][shift=({2,0})] (\x*45:0.707cm) circle [radius=2pt];
    \draw[fill=black][shift=({4,0})] (\x*45:0.707cm) circle [radius=2pt];
    }
   \draw[fill=black][shift=({0,1})] (1*45:0.707cm) circle [radius=2pt];
  \draw[thin] (0,1.35) -- ++(0,0.3);
  \draw[thin] (1,-0.35) -- ++(0,-0.3);
  
  \draw[thin] (-1.05+2,0.35+1) -- ++(0,0.3);
  \draw[thin] (-1.05+1,-0.35) -- ++(0,-0.3);
  \draw[thin] (-0.95+2,0.35+1) -- ++(0,0.3);
  \draw[thin] (-0.95+1,-0.35) -- ++(0,-0.3);
  
  \draw[thin] (1+1,0.35) -- ++(0,0.3);
  \draw[thin] (1+1,-0.35) -- ++(0,-0.3);
  \draw[thin] (1.05+1,0.35) -- ++(0,0.3);
  \draw[thin] (1.05+1,-0.35) -- ++(0,-0.3);
  \draw[thin] (0.95+1,0.35) -- ++(0,0.3);
  \draw[thin] (0.95+1,-0.35) -- ++(0,-0.3);
  
  \draw[thin] (-0.5,1.1) -- ++(-0.15,-0.15);
  \draw[thin] (-0.5,1.1) -- ++(0.15,-0.15);
  
  \draw[thin] (0.5+1,1.1) -- ++(-0.15,-0.15);
  \draw[thin] (0.5+1,1.1) -- ++(0.15,-0.15);
  
\draw[thin] (-0.5,-0.9+1) -- ++(-0.15,-0.15);
  \draw[thin] (-0.5,-0.9+1) -- ++(0.15,-0.15);
  \draw[thin] (-0.5,-1.1+1) -- ++(-0.15,-0.15);
  \draw[thin] (-0.5,-1.1+1) -- ++(0.15,-0.15);
  
  \draw[thin] (0.5+3,-0.9+1) -- ++(-0.15,-0.15);
  \draw[thin] (0.5+3,-0.9+1) -- ++(0.15,-0.15);
  \draw[thin] (0.5+3,-1.1+1) -- ++(-0.15,-0.15);
  \draw[thin] (0.5+3,-1.1+1) -- ++(0.15,-0.15);


  \node[] at (3, 0.5){$\vee$};
   \node[] at (3, -0.5){$\vee$};




\end{tikzpicture}
}

\end{minipage}
    \caption{Three examples of square-tiled surfaces. The surfaces belong to $\STS_7 \cap \calH(2,1,1)$, $\STS_5 \cap \calH(2)$, $\STS_6 \cap \calH(2)$ from left to right respectively. The surface on the left is a holonomy torus, the middle one is a visibility torus (but not a holonomy torus), and the one on the right is neither since it doesn't have a saddle connection corresponding to vector (1,0). The first two have 3 horizontal cyclinders each. The one on the right has two horizontal cylinders.}
    \label{fig:STSpermutation}
\end{figure}

Given an STS and its permutation description, one can obtain certain geometric and topological information about the surface using combinatorial properties of the pair of permutations. For instance, given $\sigma, \tau \in S_n$, the associated square-tiled surface $S(\sigma, \tau)$ is connected if and only if $\ideal{\sigma, \tau} \subseteq S_n$ is a transitive subgroup. 

The genus, and more precisely, the stratum of a square-tiled surface can also be deduced combinatorially from the associated horizontal and vertical permutations. Towards this end, consider $S= S(\sigma, \tau) \in \STS_n$ with $n$ squares. From the Euler characteristic formula, we have $\chi(S) = 2 - 2g = V - E + F$ where $V$, $E$ and $F$ is the number of vertices, edges and faces and $g$ is the genus (when $S$ is connected). Since we have $n$ squares and edges of squares are glued in pairs, $F = n$ and $E = 4n/2 = 2n$. Hence, the genus is determined by the number of vertices:
$$ g = \frac{n}{2} -\frac{V}{2} + 1 $$
Note that, more generally, when $S$ has $c$ components, $g$ can be rewritten as 

$$g = \left( \sum_{i=1}^c \frac{n_i}{2} + \frac{V_i}{2} + 1\right) - (c-1) = \left(\sum_{i=1}^c g_i \right)- (c-1)$$
where component $i$ has $n_i$ squares, $V_i$ vertices, and genus $g_i$.

The following proposition, stated in \cite{ShresRandom} relates the stratum of an STS $S (\sigma, \tau)$ with the commutator $\sigma \tau \sigma^{-1} \tau^{-1}$. Recall that a permutation $\pi \in S_n$ is said to have cycle type $(a_1, a_2, \dots, a_n)$ if $\pi$ has $a_i$ cycles of length $i$.

\begin{proposition}[Commutator to stratum data \cite{ShresRandom}] \label{prop:commutatorstratum}Let $S \in \STS_n$ be given by $(\sigma, \tau) \in S_n \times S_n$. If the cycle type of the commutator $[\sigma, \tau] = \sigma \tau \sigma^{-1} \tau^{-1}$ is $(a_1, a_2, \dots, a_n)$, then $S$ has $a_\ell$ vertices with angles $2\pi \ell$. In particular, the number of cycles of $[\sigma, \tau]$ is the number of vertices of $S(\sigma, \tau)$. 
\end{proposition}
\begin{remark}\label{rem:commutatorFixedPoint}
    A key observation in the proof of Proposition~\ref{prop:commutatorstratum} is that square $j$ is a fixed point of $[\sigma, \tau]$ if and only if the bottom left corner of $j$ is not a cone point. We will utilize this fact later. 
\end{remark}

Recall that every square-tiled surface $S$ decomposes into maximal horizontal cylinders and that these can be made up of either single bands of squares or multiple bands of squares as illustrated in Figure~\ref{fig:introSTS}. We distinguish these two cases of STSs.

\begin{definition}
A square-tiled surface $S$ is called \textbf{single-banded} if all of its maximal horizontal cylinders are made up of single bands of squares. Alternately, an STS is called single-banded if every maximal horizontal cylinder is isometric to a Euclidean cylinder of the form $(\RR/k\ZZ) \times (0, 1)$ for some $k \in \NN$. Otherwise we call $S$ $\textbf{multi-banded}$.
\end{definition}

In Section~\ref{sec:singleBanded} we give a combinatorial description of multi-banded squares and show that the asymptotic proportion of single-banded squares is 1.

For any square-tiled surface $S$, the set of holonomy vectors $\Hol(S) \subseteq \ZZ^2$ since all the vertices of the squares lie on the integer lattice and hence so do the cone points as well. Since STSs are branched covers of the square torus $\TT$ it is natural to ask when do the holonomy vectors of $\TT$ lift to STSs? Motivated by this, we have the following two types of STSs (first defined in \cite{ShresWang}).

\begin{definition}
    A square-tiled surface $S$ is called a \textbf{holonomy torus} if $\Hol(S) = \Hol(\TT)$ and a \textbf{visibility STS} if $\Hol(\TT) \subseteq \Hol(S)$.
\end{definition}

The next proposition, restated here from \cite{ShresRandom} gives combinatorial criterion on the permutations $(\sigma, \tau)$ that guarantee the associated STS $S(\sigma, \tau)$ is a holonomy torus or a visibility STS.

\begin{lemma}\label{lem:combcriterionforgeom}
    Let $\sigma$ and $\tau$ be any two permutations in $S_n$. The associated STS $S(\sigma, \tau)$ is a
    
    \begin{enumerate}
        \item \label{lem:holcriterion} holonomy torus if and only if either $[\sigma, \tau]$ is a derangement or each of the fixed points of $[\sigma, \tau]$ are also fixed by both $\sigma$ and $\tau$.
        \item \label{lem:viscriterion} visibility torus if $n > 2f$, where $f$ is the number of fixed points of $[\sigma, \tau]$.
    \end{enumerate}
\end{lemma}
\begin{proof}
    For \ref{lem:holcriterion}, see Proposition 1.6 of \cite{ShresRandom}. For \ref{lem:viscriterion}, note that Lemma 3.1 of \cite{ShresRandom} shows that $4g + 2s - 4 > n$ implies $S(\sigma, \tau)$ is a visibility torus with $g = 1- \chi(S(\sigma, \tau))/2$ where $\chi$ is the Euler characteristic and $s$ is the number of cone points of $S(\sigma, \tau)$. But $s$ is also the number of non-fixed point cycles of $[\sigma, \tau]$ and 
    $$\chi(S(\sigma, \tau)) = \text{ number of cycles of }[\sigma, \tau] - 2n + n$$
    using the Euler characteristic formula. 
    Hence, $4g+2s-4 > n \iff n > 2f$. 
\end{proof}

\subsection{Random models of square-tiled surfaces}\label{sec:randommodels}

We use the permutation description of square-tiled surfaces to give two combinatorial models for STSs. 

\textbf{Model 1 (the Standard Model)}: The first of the random models, which we call the \textbf{standard random model}, was introduced and analyzed in \cite{ShresRandom, Lech}. Since every pair of permutations $(\sigma, \tau) \in S_n \times S_n$ gives a labelled STS, the standard model consists of $S_n \times S_n$ with the uniform distribution. 

\textbf{Model 2 (the Horizontally Restricted Model)}: In the second model, which we call the \textbf{horizontally restricted (HR) model} we restrict the permutation describing how the squares are glued horizontally to a fixed conjugacy class of the symmetric group. Let $\mu_n \vdash n$ be a partition of $n$ and let $\calK_{\mu_n}$ be the conjugacy class of $S_n$ given by the partition $\mu_n$. Then the HR model consists of the set $\calK_{\mu_n} \times S_n$ with the uniform distribution. 

We introduce the notation $x \sim \Unif(X)$ to mean $x$ is chosen uniformly randomly from the set $X$. 

\begin{remark} In this note we will be concerned with the asymptotic (as $n \rightarrow \infty$) behavior of the HR model. However, note that as $n$ varies, the partition $\mu_n \vdash n$ also varies. To get a consistent asymptotic behavior, we will specify what partitions $\mu_n$ are allowed in the model as $n$ increases by restricting the number of parts of $\mu_n$. Equivalently, this restricts the number of cycles of the horizontal permutation. For instance, specifying $\mu_n \vdash n$ to have exactly 1 part will give conjugacy classes consisting of $n$-cycles. In this case, as $n$-increases, the HR model only samples from STSs with a single horizontal cylinder of height 1 and length $n$.   
\end{remark}

Given an STS $S(\sigma, \tau)$, Proposition \ref{prop:commutatorstratum} highlights the relevance of the commutator $[\sigma, \tau]$. Hence, we consider the commutator map 
$$w_c: X \times S_n \rightarrow S_n$$ given by 
$$ w_c(\sigma, \tau) = \sigma \tau \sigma^{-1} \tau^{-1}$$ where $X$ is either $S_n$ (in the case of the standard model) or $\calK_{\mu_n}$ for partition $\mu_n \vdash n$ (in the case of the HR model). Note that, since commutators are always even permutations, the image of the map $w_c$ is a subset of the alternating group $A_n \subset S_n$. Hence, the uniform distribution in $X \times S_n$ induces a probability distribution $P^X_n$ on $A_n$ as follows. For $(\sigma, \tau) \sim \Unif(X \times S_n)$, 
$$P^X_n(\pi) = \Pr(\sigma \tau \sigma^{-1} \tau^{-1} = \pi) =  \frac{|w_c^{-1}(\pi)|}{|X|\cdot n!}$$
where $w_c^{-1}(\pi) = \{ (\sigma,\tau) \in X \times S_n| [\sigma, \tau] = \pi\}$. Note that $P^X_n$ can be extended to the whole of $S_n$ by setting it to be zero on $S_n \setminus A_n$. We note that $P^X_n$ is a class function on $S_n$:
\begin{proposition}
Let $\pi, g \in S_n$. Then, $P^X_n (\pi) = P^X_n(g \pi g^{-1})$. 
\end{proposition}
\begin{proof}
Consider the sets,
\begin{align*}A &= \{(\sigma, \tau) \in X \times S_n |\,[\sigma, \tau] = \pi\}\\
B &=\{(\sigma, \tau) \in X \times S_n |\, [\sigma, \tau] = g\pi g^{-1}\} \end{align*}
Note that for any $(\sigma, \tau) \in A$, we have $(g \sigma g^{-1}, g \tau g^{-1}) \in B$ since $[g \sigma g^{-1}, g \tau g^{-1}] = g\pi g^{-1}$. Hence, there exists a map $f:A \rightarrow B$ given by $f(\sigma, \tau) = (g \sigma g^{-1}, g \tau g^{-1})$. Moreover, $f$ is a bijection with inverse $f^{-1}: B \rightarrow A$ given by $f^{-1}(\alpha, \beta) = (g^{-1} \alpha g, g^{-1} \beta g)$. Noting that $P^X_n(\pi) = \frac{|A|}{|X||S_n|}$ and $P^X_n(g \pi g^{-1}) = \frac{|B|}{|X||S_n|}$ we get the desired conclusion. 
\end{proof}

We will denote $P^X_n$ by $P_{\mu_n}$ when $X = \calK_{\mu_n}$.

\subsection{Permutation Statistics}
Our random models for square-tiled surfaces are based on permutations. Moreover, in Theorem \ref{thm:probabilityConvergenceReal}, we show that the probability distribution that the HR model induces on $A_n$ converges to the uniform distribution on $A_n$. Hence, we will naturally need a plethora of statistics on random permutations which we organize/derive in this section. 

We start with the expectation and variance of the number of cycles of a uniformly random permutation from $A_n$ or $S_n$. 
\begin{proposition}[\cite{FlemPip}]\label{prop:expectednumcycles}
    Let $C_n$ and $K_n$ be the number of cycles of a uniformly random permutation from $A_n$ and $S_n$ respectively. Then for $X_n=C_n$ or $X_n=K_n$ we have,
    $$ \EE[X_n] = \log(n) + \gamma + o(1) \text{ and } \Var[X_n] = \log(n) + \gamma - \frac{\pi^2}{6} + o(1)$$
    where $\gamma \approx 0.577$ is the Euler-Mascheroni constant. 
 
\end{proposition}

Next, we have a large deviation result on the number of cycles of $\pi \sim \Unif(A_n)$.

\begin{proposition}[\cite{FlemPip}]\label{lem:largedevforuniform}Let $\pi \sim \Unif(A_n)$ and $C_n$ be the number of cycles of $\pi$. Then, 
$$\Pr(C_n \geq t) = O\left(\frac{n}{2^t}\right)$$
\end{proposition}

Next, motivated by Lemma \ref{lem:combcriterionforgeom} we also consider the proportion of fixed point free permutations (also known as \textbf{derangements}) in $A_n$.

\begin{proposition}[\cite{PoonSlav}]\label{prop:AnDerangement}
    Let $\pi \sim \Unif(A_n)$. Then, $|\Pr(\pi \text{ is a derangement}) - \frac{1}{e}| \leq \frac{n^2}{(n+1)!}$.
\end{proposition}

Let $\pi \sim \Unif(A_n)$, let $g \in A_n,$ and let $\mathcal{K}_g$ be the conjugacy class of $g$ in $S_n.$ Then the probability that $\pi$ is in $\mathcal{K}_g$ is \begin{equation*}\text{Pr}(\pi \in \mathcal{K}_g) = \frac{|\calK_g|}{|A_n|} =  \frac{2 |\mathcal{K}_g|}{n!}.\end{equation*}

Recall that a permutation $\sigma \in S_n$ is said to have $\textbf{cycle type}$ $(a_1, \dots, a_n)$ if it has $a_1$ 1-cycles, $a_2$ 2-cycles, $\dots$, $a_n$ $n$-cycles. Moreover, any two permutations are conjugate in $S_n$ if and only if they have the same cycle type. Given a conjugacy class $\calK$ and $\sigma \in \calK$, we say  $\sigma$ is the \textbf{canonical representative of} $\mathcal{K}$ if $\sigma$ has the form $(1 \cdots m_1) (m_1 + 1 \cdots m_2) (m_2 + 1 \cdots m_3) \cdots (m_{k-1}+1 \cdots m_{k}),$ with $1 = m_0 \leq m_1 < m_2 < \cdots < m_k = n$ and $m_i - m_{i - 1} \leq m_{i - 1} - m_{i - 2}$ for all $i.$ In other words, the largest cycles appear first and the numbers $1 \dots n$ appear in order. We also say that a cycle is $\textbf{non-trivial}$ if it is not a 1-cycle. 

The next well known result gives concrete expressions for $|\mathcal{K}_g|$ and $\text{Pr}(\pi \in \mathcal{K}_g)$ in terms of the cycle type of $g$.

\begin{proposition}\label{prop:cycleNumbers}
    Let $g \in A_n$ have cycle type $t = (a_1, a_2, \ldots, a_n)$. Let $\mathcal{K}_g$ be the conjugacy class of $g$ in $S_n$. Then, \begin{equation}|\mathcal{K}_g| = n! \left(\prod_{k = 1}^n \frac{k^{-a_k}}{a_k !}\right).\label{eq:classSize}\end{equation} In particular, for $\pi \sim \Unif(A_n)$, we have \begin{equation*}\text{Pr}(\pi \in \mathcal{K}_g) =2 \left(\prod_{k = 1}^n \frac{k^{-a_k}}{a_k !}\right).\end{equation*}
\end{proposition}

\begin{proof}
    Let $\vec{d}$ be the tuple corresponding to the cycle type $t$ where $\vec{d}$ contains $a_1$ repeating 1's, $a_2$ repeating 2's, $a_3$ repeating 3's, and so on. Then recall that the multinomial coefficient $\binom{n}{\vec{d}}$ represents the number of ways of sorting the digits $1$ through $n$ into $a_1$ bins of size 1, $a_2$ bins of size 2, $a_3$ bins of size 3, and so on. Given a bin with $k$ numbers, one can form $(k - 1)!$ distinct $k$-cycles. But given $a_k$ disjoint $k$-cycles, we can order them in $a_k !$ ways. Hence, noticing that a permutation in $\calK_g$ has $a_k$ disjoint $k$-cycles, we see
    
    \begin{equation*} |\mathcal{K}_g| = \binom{n}{\vec{d}} \prod_{k = 1}^n \frac{((k - 1)!)^{a_k}}{a_k !} = \frac{n!}{\prod_{k=1}^n (k!)^{a_k}}\prod_{k = 1}^n \frac{((k - 1)!)^{a_k}}{a_k !}  =n! \left(\prod_{k = 1}^n \frac{k^{-a_k}}{a_k !}\right).\end{equation*} Thus, \begin{equation*}\text{Pr}(\pi \in \mathcal{K}_g) = \frac{|\calK_g|}{|A_n|}= 2 \left(\prod_{k = 1}^n \frac{k^{-a_k}}{a_k !}\right).\end{equation*}\end{proof}

\subsection{Tools from representation theory}

Our proofs rely heavily on tools from the representation theory of finite groups, and in particular, symmetric groups. In this section we give a brief introduction to the representation theory of finite groups by citing some foundational definitions and results. We also state more recent key theorems that we will subsequently use in the latter half of the paper. See \cite{Isaacs} for a full reference.

Suppose $G$ is a finite group, let $F$ be a field and $V$ a vector space over $F$ of dimension $n$. An \textbf{$F$-representation} is a group homomorphism $\rho : G \rightarrow \text{GL}(V)$. Choosing a basis for $V$, we can identify $\text{GL}(V)$ with $\text{GL}(n; F)$, the group of $n \times n$ invertible matrices over $F$. So, a representation can be thought of as a homomorphism $\rho: G \rightarrow \text{GL}(n;F)$. The \textbf{character} afforded to representation $\rho$ is the map $\chi : G \rightarrow F$ given by $\chi(g) = \text{Tr}(\rho(g))$ The \textbf{dimension} or \textbf{degree} of the representation/character is the dimension of $V$ (in this case, $n$). Two representations $\rho_1, \rho_2$ of the same dimension are \textbf{similar} or \textbf{equivalent}  if there exists an invertible matrix $P$ such that $\rho_1(g) = P \rho_2(g) P^{-1}$ for all $g \in G$. It is well known that two representations are similar if and only if they afford the same character.

Given a group $G$ and a field $F$, denote by $F[G]$ the group algebra of $G$ over $F$. Note that a representation $\rho: G \rightarrow \text{GL}(n; F)$ can be naturally extended to a representation $\rho: F[G] \rightarrow M_n(F)$ (where $M_n(F)$ is the set of $n\times n$ matrices over $F$) and vice versa. 

Given a representation $\rho: G \rightarrow \text{GL}(V)$, a subspace $W \subset V$ is called \textbf{$G$-invariant} if $\rho(g)w \in W$ for all $g \in G$ and $w \in W$. The representation $\rho$ is said to be \textbf{irreducible} if there are no non-trivial $G$-invariant subspaces $W\subset V$. A character $\chi$ is said to be irreducible if the corresponding representation it is afforded to is irreducible. 

For our purpose we will be focused on the case when $F = \CC$. It is a well known theorem by Maschke that one can always break up a representation into irreducible representations. This allows us to study the irreducible pieces and then recover information about the whole. 

There are a number of important facts about the characters corresponding to irreducible representations. 
\begin{theorem}\label{thm:characterprops}
    Suppose $G$ is a finite group and let $\Irr(G)$ be the set of irreducible characters of $G$. Then we have the following facts about characters of $G$:
    \begin{enumerate}
        \item The set $\Irr(G)$ is a finite group (under composition) and its cardinality is equal to the number of conjugacy classes of $G.$
        \item If $g, h \in G$ are conjugate, then $\chi(g) = \chi(h)$ for any irreducible character $\chi$. In other words, irreducible characters are $\textbf{class functions}$.
        \item (First orthogonality) If $\chi_i$ and $\chi_j$ are irreducible characters of $G.$ Then \begin{equation}\frac{1}{|G|} \sum_{g \in G} \chi_i(g) \overline{\chi_j(g)} = \delta_{ij},\end{equation} where $\delta_{ij} = 1$ if $i = j$, and $\delta_{ij} = 0$ otherwise.
        \item (Second orthogonality) Suppose $g, h \in G$. Then \begin{equation}\sum_{\chi \in \Irr(G)} \chi(g) \overline{\chi(h)} = \begin{cases}|G|/|\mathcal{K}_g|, & g\text{ is conjugate to } h \\ 0, & \text{else}\end{cases},\end{equation} where $|\mathcal{K}_g|$ is the size of the conjugacy class of $g.$
    \end{enumerate}
\end{theorem}

Additionally, we need the following lemma (Lemma 2.25 of \cite{Isaacs}) which severely restricts the form of commuting representations:

\begin{lemma}\label{lem:diagonalmatrix}
    Suppose $\rho$ is an $n$-dimensional irreducible $\mathbb{C}$-representation of a group $G$. Suppose also that $A \in M_n(\CC)$ commutes with $\rho(g)$ for all $g \in G.$ Then $A = \alpha I$ for some $\alpha \in \mathbb{C},$ where $I$ is the $n \times n$ identity matrix. 
\end{lemma}


Note that if $z$ is in the center $Z(\CC[G])$ of the group algebra, then $z$ commutes with everything in $G$. As any representation $\rho$ is a homomorphism, $\rho(z)$ must commute with $\rho(g)$ for every $g \in G$ so that $\rho(z) = \alpha I$ (by Lemma~\ref{lem:diagonalmatrix}). This allows us to define the following function:

\begin{definition}\label{def:omegaFunction}Given $\chi \in \Irr(G)$, let $\rho$  be an irreducible representation affording $\chi$. Define a map $\omega_\chi: Z(\CC[G]) \rightarrow \CC$ from the center of $\CC[G]$ to $\CC$ by setting 
$$\omega_\chi(z) = \alpha$$
where $\alpha$ is such that $\rho(z) = \alpha I$.    
\end{definition}
Alternately, we define $\omega_\chi(z)$ for $z \in Z(\CC[G])$ such that $\rho(z) = \omega_\chi(z) I.$ The next proposition asserts that given $\chi$ an irreducible representation, $\omega_\chi$ is independent of the representation affording $\chi$:
%

\begin{proposition}
   Given an irreducible character $\chi \in \Irr(G)$, let $\omega_\chi : Z(\mathbb{C}[G]) \rightarrow \mathbb{C}$ be as above. Then $\omega_{\chi}$ is a well defined homomorphism. 
\end{proposition}

\begin{proof}
    To see that $\omega_\chi$ is well defined, consider two representations $\rho$ and $\rho'$ affording $\chi$. For $z \in Z(\CC[G])$, let $\alpha, \alpha' \in \CC$ such that $\rho(z) = \alpha I$ and $\rho'(z) = \alpha' I$. But since $\rho'$ is equivalent to $\rho$, $\alpha' I$ is similar to $\alpha I$ which implies $\alpha' = \alpha$. In addition, since $\rho$ is a homomorphism, $\omega_\chi$ is as well. \end{proof}

To apply the homomorphism $\omega_\chi$, a basis for $Z(\CC[G])$ will be highly relevant. Towards this end, we have the following definition:
\begin{definition}\label{def:classSum}
    Let $\calK$ be a conjugacy class of a finite group $G$ and let $K:=\sum_{g \in \calK} g \in \CC[G]$ be the \textbf{class sum over the conjugacy class} $\calK.$
\end{definition}

First note given a class sum $K$, $h K h^{-1} = K$ for any $h \in \CC[G]$. Hence, any class sum is in the center $Z(\CC[G])$. Furthermore, it is known (see, for instance, Theorem 2.4 of \cite{Isaacs}) that these class sums form a basis for $Z(\CC[G])$:

\begin{proposition}\label{prop:classsumbasis}
Let $\calK_1, \dots ,\calK_n$ be conjugacy classes of $G$. Let $K_i = \sum_{g \in \calK_i} g$. The set $\{K_1, \dots, K_n\}$ forms a basis for $\ZZ(\CC[G])$. Moreover, if $K_i K_j = \sum_{k} a_{ijk} K_k$, then the $a_{ijk}$ are non-negative integers.  
     
\end{proposition}

Additionally, we can express $\omega_\chi(K)$ in terms of the character $\chi$ for any class sum $K \in Z(\CC[G])$:

\begin{proposition}\label{prop:omegachi}
    If $\calK$ is a conjugacy class of $G$ and $K:=\sum_{g \in \calK} g \in Z(\CC[G])$ is the class sum over the conjugacy class $\calK$, then $\omega_\chi(K) = \frac{\chi(g) |\calK|}{\chi(1)},$ where $g$ is an arbitrary element of $\calK.$
\end{proposition}
\begin{proof}
  Let $\rho$ be any representation affording $\chi$ is a character. By definition of $\omega_\chi$ we have \begin{equation*}\rho(K) = \omega_\chi(K) I.\end{equation*} Taking the trace of both sides and noting irreducible characters are class functions, we get
  $$\sum_{h \in \calK}\chi(h) = \omega_\chi(K)\chi(1) \implies |\calK| \chi(g) = \omega_\chi(K) \chi(1)$$
  where $g \in \calK$ is arbitrary. This implies $\omega_\chi(K) = \frac{\chi(g)|\calK|}{\chi(1)}$.
\end{proof}

We will also need a few more tools from probability and Fourier transforms in the setting of finite groups. We start by defining a distance between two probability distributions defined on a finite space:

\begin{definition}
    Suppose that $P$ and $Q$ are two probability distributions on a finite space $Y$. Then the \textbf{max norm difference} or the \textbf{total variation distance} of $P$ and $Q$, denoted $||P - Q||$ is given by \begin{equation*}||P - Q|| = \max_{B \subseteq Y} |P(B) - Q(B)|.\end{equation*}
\end{definition} 
It is known that the total variation distance between $P$ and $Q$ can alternately be expressed as
$$ ||P-Q|| = \frac{1}{2} \sum_{y \in Y} |P(y) - Q(y)|$$

We also use the notion of the Fourier transform of a probability measure. In our setting of finite groups, it takes the following form:

\begin{definition}
The \textbf{Fourier transform} of a probability measure $P$ on a finite group $G$ at a representation $\rho$ is defined to be
$$\hat{P}(\rho) = \sum_{g \in G} P(g) \rho(g)$$    
\end{definition}

All of the probability distributions in our setting will be constant on conjugacy classes. In such a case, Diaconis-Shahshahani showed that the Fourier transform takes a special form:

\begin{lemma}[Diaconis-Shahshahani \cite{DiaShah}] \label{lem:ds}Let $P$ be a probability measure on a finite group $G$ that is constant on conjugacy classes of $G$. Let $\rho$ be a linear representation of $G$. The Fourier transform of $P$ at $\rho$ takes on the special form,

$$\hat{P}(\rho) = \frac{1}{\dim(\rho)} \sum_{K} P(K) \cdot |K| \cdot \chi^\rho(K) \cdot I_{\dim (\rho)} $$

where 
\begin{enumerate}
    \item $\dim(\rho)$ is the dimension of $\rho$, 
    \item the sum is over conjugacy classes of $G$,
    \item $\chi^\rho$ is the character afforded by $\rho$, and 
    \item $I_{\dim(\rho)}$ is the identity matrix of dimension $\dim(\rho)$.
\end{enumerate}

\end{lemma}

In our proofs we will often be required to estimate the magnitude of irreducible characters. We state here two such key estimates. The first one is due to Larsen-Shalev:

\begin{theorem}[Larsen-Shalev \cite{LarSha}]\label{thm:larsenshalevbound} Fix $\alpha \in [0, 1]$. For all $\epsilon > 0$, there exists $N \in \NN$ such that if $n > N$ and $\sigma \in S_n$ with at most $n^\alpha$ cycles, then for any irreducible character $\chi \in \Irr(S_n)$, 
$$ |\chi(\sigma)| \leq \chi(1)^{\alpha + \epsilon}.$$

\end{theorem}
The other estimate is due to Liebeck-Shalev:
\begin{theorem}[Liebeck-Shalev \cite{LieSha}]\label{thm:liebeckshalev} Let $s \in \RR_>0$ be fixed. Then,

$$\sum_{\chi \in \Irr(S_n)} \frac{1}{\chi(1)^s} = 2 + O\left(\frac{1}{n^s}\right).$$
\end{theorem}

\subsection{Tools from Combinatorics}

The random models we use are based on permutations and the representation theory of symmetric groups. These objects are intimately understood using tools from combinatorics -- in particular, the theory of partitions and symmetric functions. In this section we give a brief exposition on these topics relevant for our proofs. For a full reference to the content in this section, see \cite{Stan1, Stan2}.

\subsubsection{Partitions, Young Tableau and the Murnaghan-Nakayama Rule}


For $n \in \NN$, a \textbf{partition} $\lambda$ of $n$, denoted $\lambda \vdash n$ is a sequence  $\lambda=(\lambda_1, \dots, \lambda_k) \in \NN^k$ such that $\sum \lambda_i = n$ and $\lambda_1 \geq \lambda_2 \geq \dots \geq \lambda_k$. 
Two partitions of $n$ are identical if they only differ in the number of zeros. For example $(3, 2, 1, 1) = (3, 2, 1, 1, 0, 0)$. Informally a partition of $n$ can be thought of as a way of writing $n$ as a sum of positive integers,  disregarding the order of the integers. The non-zero $\lambda_i$ are called \textbf{parts} of $\lambda$. If $\lambda$ has $\xi_i$ parts equal to $i$, then we can write $\lambda = \ideal{1^{\xi_1}, 2^{\xi_2}, \dots}$ where terms with $\xi_i = 0$ and the superscript $\xi_i = 1$ is omitted. The set of all partitions of $n$ is denoted by $\Par(n)$ and let $\Par := \bigcup_{n \in \NN} \Par(n)$.

\begin{definition}[Young diagram]\label{def:youngdiagram}
    A partition $(\lambda_1, \dots, \lambda_k) \vdash n$ can be visually represented by a \textbf{Young diagram}, which is a left-justified array of boxes with $\lambda_i$ boxes in the $i$-th row.
\end{definition} 
See Figure~\ref{fig:youngdiagram} for an example illustrating a Young diagram (together with some other terms that will be defined shortly). The squares in a Young diagram are identified using tuples $(i,j)$ where $i$ is the row corresponding to part $\lambda_i$ and $1 \leq j\leq \lambda_i$ is the position of the square along that row. 
\begin{definition}[Hook length and content]\label{def:hooklength}
   Given a square $r = (i,j) \in \lambda$, the \textbf{hook length of $\lambda$ at $r$}, denoted $h(r)$, is the number of squares directly to the right or directly below $r$, counting $r$ itself once. The \textbf{hook length product of $\lambda$}, denoted $H_\lambda$ is the product, 
$$H_\lambda = \prod_{r \in \lambda} h(r)$$ 
The \textbf{content} of $\lambda$ at $r = (i, j)$ is the integer $\cont(r) = j-i$.
\end{definition}

More generally, the boxes of the Young diagram can be filled with entries from a totally ordered set (usually a set of positive integers) to obtain a \textbf{Young tableau}. 

\begin{definition}[Semi-standard Young tableau and border-strip tableau]\label{def:SSYTBST} A Young tableau is called a \textbf{semistandard Young tableau (SSYT)} if the entries of the tableau weakly increase along each row and strictly increase down each column. A \textbf{border-strip tableau (BST)} is a Young tableau where the entries in every row and column are weakly increasing, and the set of squares filled with the integer $i$ (called a \textbf{border-strip}) form a connected shape with no 2×2-square. The \textbf{type of a SSYT or BST} is a sequence $\xi = (\xi_1, \dots )$ where the SSYT or BST contains $\xi_1 1's$, $\xi_2 \, 2's$ and so on. The \textbf{shape of an SSYT or BST} is the associated partition.
    
\end{definition}

\begin{definition}[Height]\label{def:height}
    In a border-strip tableau, the \textbf{height of a border-strip} is one less than the number of rows it touches. The \textbf{height of a border-strip tableau} $T$, denoted $h(T)$, is the sum of the heights of its border-strips. 
\end{definition}

\begin{figure}[h!!]
\begin{minipage}{0.19\linewidth}
\centering
$\begin{ytableau} {}&&&& \cr {}&& \cr {}&& \cr {}& \cr {}
\end{ytableau}$
\end{minipage}
\begin{minipage}{0.19\linewidth}
\centering
$\begin{ytableau} {9}&{7}&{5}&{2}&{1}\cr {6}&{4}&{2} \cr {5}&{3}&{1} \cr {3}&{1} \cr {1}
\end{ytableau}$
\end{minipage}
\begin{minipage}{0.19\linewidth}
\centering
$\begin{ytableau} {0}&{1}&{2}&{3}&{4} \cr {-1}&{0}&{1} \cr {-2}&{-1}&{0} \cr {-3}&{-2} \cr {-4}
\end{ytableau}$
\end{minipage}
\begin{minipage}{0.19\linewidth}
\centering
$\begin{ytableau} {1}&{1}&{2}&{3}&{3} \cr {2}&{2}&{3} \cr {5}&{5}&{5} \cr {6}&{8} \cr {7}
\end{ytableau}$
\end{minipage}
\begin{minipage}{0.19\linewidth}
\centering
$\begin{ytableau} {1}&{1}&{3}&{3}&{3} \cr {2}&{2}&{3} \cr {5}&{5}&{5} \cr {6}&{8} \cr {7}
\end{ytableau}$
\end{minipage}
\caption{Young diagram and tableau for the partition (5,3,3,2,1). From left to right: Young diagram;  Tableau with hook lengths; Tableau with contents; Example of an SSYT that is not a BST (since, for instance, the squares filled with 2 is not connected); Example of a BST. The height is $h(T) = 0+0+1+0+0+0+0 = 1$ (since the only border-strip that touches two rows is the one corresponding to squares filled with 3).}
\label{fig:youngdiagram}
\end{figure}
Note that given a permuation $\pi \in S_n$, the cycle type of $\pi$ is a partition of $n$. Moreover, two permutations are conjugate if and only if they have the same cycle type. Hence, partitions index conjugacy classes of $S_n$ and we denote $\cyc(\pi)$ to mean the cycle type and the partition simultaneously. Additionally, it is known that partitions of $n$ also index irreducible representations and irreducible characters of $S_n$ canonically. Hence, we will denote by $\rho^\lambda$ the irreducible representation indexed by $\lambda$ and by $\chi^\lambda$ the character of the representation $\rho^\lambda$, i.e. $\chi^\lambda(\pi) = \Tr(\rho^\lambda(\pi))$ for any $\pi\in S_n$. In particular, note that the dimension, $\dim(\rho^\lambda) = \chi^\lambda(1)$ where $1 \in S_n$ denotes the identity permutation.  

One instance of the deep relation between partitions and the representation theory of the symmetric group is the Hook-length formula which gives a relation between the degree (or dimension)  of a representation and the hook length product of the partition that is used to index it:

\begin{lemma}[Hook-length formula]\label{lem:hooklength}
For a partition $\lambda \vdash n$, and associated representation $\rho^\lambda$, the dimension of $\rho^\lambda$ is given by $$ \dim(\rho^\lambda) = \frac{n!}{H_\lambda}$$
\end{lemma}

Using the hook-length formula and noting that the dimension of a representation is given by the associated character evaluated at the identity, we can get the following estimate which we will utilize in the proof of Lemma \ref{lem:largedeviations}.

\begin{proposition}\label{prop:dimensionbound}

For $1 \leq k \leq \floor{n/2}$, let $(n-k,k)$ be the partition of $n$ with parts $n-k$ and $k$. Then,
$$\chi^{(n-k, k)}(1) = {n \choose k} \frac{n-k-(k-1)}{n-k+1} \leq {n \choose k}$$ 
In particular, the dimension of the irreducible representation associated to $(n-k,k)$ is less than or equal to $ {n \choose k}$.

\end{proposition}
\begin{proof}
Let $r = (i,j)$ be a square in the Young diagram of $(n-k, k)$. If $i =1$ and $j \leq k$, the hook length $h(r) = 1 + 1 + (n-k-j)$ since there is 1 square below $r$ and $n-k-j$ to the right of $r$. If $i = 1$ and $j > k$, then $h(r) = 1 + (n-k-j)$ since there are no squares below $r$ and $n-k-j$ squares to the right of $r$. If $i = 2$, then $h(r) = 1+(k-j)$ since there are no squares below it and $k-j$ squares to the right. 

Putting all of these hook lengths together, we get the hook length product of $(n-k,k)$:
\begin{align*}
    H_{(n-k, k)} &= \prod_{j=1}^{k} (n-k-j+2) \prod_{j=k+1}^{n-k} (n-k-j+1) \prod_{j=1}^k(k-j+1)\\
    &= \frac{(n-k+1)!k!}{(n-k-(k-1))}.
\end{align*}
So, using the hook-length formula we have,
\begin{align*}
    \chi^{(n-k, k)}(1) &= \frac{n!(n-k-(k-1))}{(n-k+1)!k!}\\
    &= \frac{n!}{(n-k)!k!} \cdot \frac{1}{n-k+1}\cdot (n-k-(k-1))\\
    &={n \choose k} \frac{n-k-(k-1)}{n-k+1} \leq {n \choose k} 
\end{align*}\end{proof}

Another important and powerful formula that allows one to use the theory of partitions to evaluate the value of irreducible characters of $S_n$ on conjugacy classes is the Murnaghan-Nakayama Rule. Recall that irreducible characters are class functions (i.e. constant on conjugacy classes) and that conjugacy classes in the symmetric group are indexed by partitions. Also recall Definitions~\ref{def:SSYTBST} and \ref{def:height} for definitions shape, type and height of a border-strip tableau.

\begin{theorem}[Murnaghan-Nakayama Rule]\label{thm:murnaghannakayama}
 Let $\lambda$ and $\mu_n$ be two partitions of $n$. Then, the character associated to $\lambda$ evaluated at any element of the conjugacy class $\calK_{\mu_n}$ is given by
\begin{equation}\label{eq:murnaghannakayama}\chi^\lambda (\calK_{\mu_n}) = \sum_{\substack{T \text{ is BST of}\\ \text{shape } \lambda, \text{ type }\mu_n}}(-1)^{ht(T)}\end{equation}
where the sum is over all border-strip tableau (BST) of shape $\lambda$ and type $\mu_n$, and $ht(T)$ denotes the height of the tableau $T$. 
\end{theorem}
The Murnaghan-Nakayama Rule will be used in Lemma~\ref{lem:onecylindercharacter} to evaluate possible irreducible character values evaluated on $n$-cycles.

\subsubsection{Power sum symmetric functions and Schur functions}

Let $x = (x_1, x_2,\dots)$ be a set of indeterminates (i.e. formal placeholder variables), and let $n \in \NN$. 

Denote the set of all homogeneous symmetric functions of degree $n$ over $\QQ$ as $\Lambda^n$. It is known than $\Lambda^n$ is a $\QQ$-vector space. We will utilize two standard bases for $\Lambda^n$ -- power sum symmetric functions and Schur functions. 

The \textbf{power sum symmetric functions}, denoted $p_\lambda$ are indexed by partitions $\lambda$ and defined as, 
\begin{align*}
&p_m := \sum_{i} x_i^m, \hspace{1cm} m \geq 1 \hspace{1cm} \text{ (with }p_0 = 1)\\
&p_\lambda := p_{\lambda_1}p_{\lambda_2}\dots p_{\lambda_k} \hspace{1cm} \text{ if }\lambda = (\lambda_1, \lambda_2, \dots, \lambda_k ) \vdash n
\end{align*}

The \textbf{Schur functions}, denoted $s_\lambda$,  are also indexed by partitions $\lambda$. They are defined as the formal power series,
$$ s_\lambda(x) = \sum_{T} x^T$$
where
\begin{itemize}
\item the sum is over all semi-standard Young tableau (SSYT) $T$ of shape $\lambda$ (i.e. those SSYTs whose associated partition of $n$ is $\lambda$).
\item $x^T := x_1^{\xi_1(T)} x_2^{\xi_2(T)} \dots$ if $T$ is an SSYT of type $\xi$ (i.e. if $T$ contains $\xi_1(T)$ 1's, $\xi_2(T)$ 2's, etc.)
\end{itemize}
While it is not hard to see that power sum symmetric functions are indeed symmetric functions, it is a non-trivial theorem that Schur functions  are indeed symmetric functions. 
The change of basis formula to go from Schur functions $s_\lambda$ to power sum symmetric functions $p_\lambda$ exposes the deep relation these objects have with the irreducible characters of $S_n$. 

\begin{theorem}[Expressing a Schur function as a sum of power sum symmetric functions] For $\lambda \vdash n$,
$$ s_\lambda = \sum_{\nu \vdash n} \frac{|\calK_\nu|}{n!} \chi^\lambda (\calK_\nu)p_\nu$$
where the sum is over partitions $\nu$ of $n$, $\chi^\lambda$ is the irreducible character of $S_n$ indexed by $\lambda$, $\calK_\nu$ is the conjugacy class of $S_n$ with cycle type $\nu$.    
\end{theorem}

Conversely, the change of basis formula from power sum symmetric functions to Schur functions is as follows:

\begin{theorem}[Expressing a power sum symmetric function as a sum of Schur functions]\label{thm:powersumtoschur} The power sum symmetric functions can be expressed as a linear combination of the Schur functions in the following way:
$$ p_\mu = \sum_{\lambda \vdash n} \chi^\lambda (\calK_{\mu_n}) s_\lambda$$
\end{theorem}

We state a particular specialization of $s_\lambda$ which we will use. This is Corollary 7.21.4 of \cite{Stan2}. Recall that $\Par = \bigcup_{n \in \NN} \Par(n)$ denotes the set of all partitions. Also recall from Definition~\ref{def:hooklength} that for $r = (i,j)$ in a Young diagram, $\cont(r) = j-i$ is the content of the square $r$ and $h(r)$ is the hook length of $r$ (i.e. the number of squares directly to the right or directly below $r$, counting $r$ itself once). 

\begin{lemma}\label{lem:schurspec} For any $\lambda \in \Par$ and $m$ a positive integer, we have

$$s_\lambda(1^m) = \prod_{r \in \lambda} \frac{m+\cont(r)}{h(r)}$$ 
where $s_\lambda(1^m)$ means evaluating $s_\lambda$ by setting $x_1 = \dots x_m = 1$ and $x_i = 0$ for all $i > m$.
\end{lemma}

Next, let $\CF^n$ be the set of class functions $f: S_n \rightarrow \QQ$. Then, there exists a natural inner product on $\CF^n$ given by,
$$ \ideal{f, g} = \frac{1}{n!}\sum_{\sigma \in S_n} f(\sigma)g(\sigma) = \frac{1}{n!}\sum_{\lambda\vdash n} |\calK_\lambda| f(\lambda)g(\lambda)$$
where $f(\lambda)$ denotes the value of $f$ on the conjugacy class associated to partition $\lambda$. To each class function $f \in \CF^n$, one can associate a symmetric function of degree $n$ via the linear transformation $\ch: \CF^n \rightarrow \Lambda^n$, called the \textbf{Frobenius characteristic map} given by,

\begin{equation}\label{eq:frobeniuscharmap1} \ch f = \frac{1}{n!}\sum_{\sigma \in S_n} f(\sigma)p_{\cyc(\sigma)} = \frac{1}{n!} \sum_{\lambda \vdash n} |\calK_\lambda| f(\lambda) p_\lambda.
\end{equation} Using Theorem \ref{thm:powersumtoschur}, the characteristic function is expressed as,
\begin{equation}\label{eq:frobeniuscharmap2} \ch f = \sum_{\lambda \vdash n} \ideal{f, \chi^\lambda}s_\lambda
\end{equation}

\section{Proportion of single-banded surfaces}\label{sec:singleBanded}

In this section we prove that, in the Standard model, the probability that a random STS is single-banded is asymptotically 1. Recall that for $(\sigma,\tau) \in S_n \times S_n$, $S(\sigma, \tau)$ represents the square-tiled surface with permutations $\sigma$ and $\tau$ that describe how the squares are glued horizontally and vertically. 

\singlebanded*

Before we prove Theorem~\ref{thm:singleBanded}, we need a few lemmas. The first one characterizes multi-banded surfaces combinatorially:

\begin{lemma}\label{lem:multibandCharacterization}
    Let $S(\sigma, \tau) \in \STS_n$. The surface $S(\sigma, \tau)$ is multi-banded if and only $\sigma$ contains disjoint cycles $(a_0, \dots, a_{k-1})$ and $(b_0, \dots, b_{k-1})$ of the same length such that for some $i \in \{0, \dots, k-1\}$, $\tau(b_{i+j}) = a_j$ for all $j = 0,\dots, k$, where the indices of $a$ and $b$ are taken modulo $k$. 
\end{lemma}
\begin{proof}
   $(\Rightarrow)$ Assume $\sigma$ contains disjoint cycles $(a_0, \dots, a_{k-1})$ and $(b_0, \dots, b_{k-1})$ such that for some $i \in \{0, \dots, k-1\}$, $\tau(b_{i+j}) = a_j$ for all $j = 0\dots k$. We will show that the squares labeled $a_0, \dots, a_{k-1}$ and $b_0, \dots, b_{k-1}$ form an embedded (possibly non-maximal) cylinder $(\RR/k\ZZ) \times (0, 2)$ in the surface. Hence, the maximal cylinder containing it cannot be formed using a single band of squares.  
    
    First, note that since $(a_0, \dots, a_{k-1})$ and $(b_0, \dots, b_{k-1})$ are cycles of $\sigma$, the squares $a_0, \dots, a_{k-1}$ and $b_0, \dots, b_{k-1}$ each (individually) form cylinders of the form $(\RR/k\ZZ) \times (0,1)$. Note that the ordered concatenation of the bottom edges of $a_0, \dots, a_{k-1}$ (respectively the top edges of $b_0, \dots, b_{k-1}$) form closed curves $C_1$ (respectively $C_2$) on the surface. See Figure~\ref{fig:MultibandCharacterization}.
\begin{figure}[h!]
\resizebox{6.5cm}{!}{
\begin{tikzpicture}
    [sq/.style=
  {shape=regular polygon, regular polygon sides=4, draw, minimum width=1.414cm,inner sep=0pt}, rect/.style={shape=regular polygon, regular polygon sides=4, draw, minimum width=1pt}]
\node[sq] at (0,0){$b_3$};
\node[sq] at (1,0){$b_4$};
\node[sq] at (2,0){$\dots$};
\node[sq] at (3,0){$\dots$};
\node[sq] at (4,0){$\dots$};
\node[sq] at (5,0){$b_{k-1}$};
\node[sq] at (6,0){$b_1$};
\node[sq] at (7,0){$b_2$};

\node[sq] at (0,2){$a_0$};
\node[sq] at (1,2){$a_1$};
\node[sq] at (2,2){$a_2$};
\node[sq] at (3,2){$\dots$};
\node[sq] at (4,2){$\dots$};
\node[sq] at (5,2){$\dots$};
\node[sq] at (6,2){$a_{k-2}$};
\node[sq] at (7,2){$a_{k-1}$};


\draw[ultra thick, blue] (-0.5, 0.5) -- (7.5, 0.5);
\node[] at (3.5, 0.75){$C_2$};

\draw[ultra thick, red] (-0.5, 1.5) -- (7.5, 1.5);
\node[] at (3.5, 1.25){$C_1$};


\node[] at (-0.5, 0){$\sim$};
\node[] at (7.5, 0){$\sim$};

\node[rotate=90] at (-0.5, 2){$\vee$};
\node[rotate=90] at (7.5, 2){$\vee$};

\end{tikzpicture}
}
\caption{The cycles $(a_0, \dots, a_{k-1})$ and $(b_0, \dots, b_{k-1})$ form two bands of squares. In the schematic depicted, $b_{3+j}$ is glued to $a_{j}$ for $j = \{0, \dots, k-1\}$ where the indices of are taken modulo $k$.}
\label{fig:MultibandCharacterization}
\end{figure}

    Moreover, the bottom edge of square $a_j$ is glued to the top edge of $b_{i+j}$ for every $j$. This means $C_1$ and $C_2$ are glued to each other isometrically. Hence, to show that the squares $a_0, \dots, a_{k-1}$ and $b_0, \dots, b_{k-1}$ together form an annulus of the form $(\RR/k\ZZ) \times (0,2)$, it suffices to show the curve $C_1 \sim C_2$ where the two annuli of squares are glued does not contain any cone points. 

    To see this, let $v$ be a square corner on $C_1$. Then $v$ is the bottom left corner of a square $a_j$ for some $j$. We then compute that 
    $$[\sigma, \tau](a_j) = (\sigma\tau\sigma^{-1}\tau^{-1})(a_j) =(\sigma\tau\sigma^{-1} b_{i+j} = (\sigma\tau) b_{i+j-1} = \sigma(a_{j-1}) = a_j.$$ Hence, as $a_j$ is a fixed point of the commutator $[\sigma, \tau]$, by Remark~\ref{rem:commutatorFixedPoint}, $v$ is not a cone-point.

    $(\Leftarrow)$ Assume $S(\sigma, \tau)$ has a maximal horizontal cylinder comprised of multiple bands of squares. Let $k \in \NN$ be the length of the cylinder and $h\in \NN$ the height. Then, by assumption $h \in \NN$ and $h \geq 2$. Let $a_0, \dots, a_k$ be the squares on the top of the cylinder from left to right. By definition of $\sigma$, $(a_0, \dots, a_k)$ is a cycle of $\sigma$. Consider squares $\tau^{-1}(a_0), \dots, \tau^{-1}(a_{k-1})$. We claim that $(b_0, \dots, b_{k-1}):= (\tau^{-1}(a_0), \dots, \tau^{-1}(a_{k-1}))$ is a cycle of $\sigma$.

    To see that $(b_0, \dots, b_{k-1})$ is a cycle of $\sigma$, consider $b_j = \tau^{-1}(a_j)$ for some $j \in \{0, \dots, k-1\}$. We will show that $\sigma^{-1}(b_j) = b_{j-1}$ (where the index is taken modulo $k$).

    Towards this end, 
    $$\sigma^{-1}(b_j) = \sigma^{-1}\tau^{-1}(a_j)$$
    But the bottom edge of $a_j$ is in the interior of the multi-banded maximal cylinder. So, the lower left corner of $a_j$ is not a cone point. By Remark~\ref{rem:commutatorFixedPoint}, $a_j$ is a fixed point of $[\sigma, \tau]= \sigma\tau\sigma^{-1}\tau^{-1}$. So,

    $$\sigma^{-1}\tau^{-1}(a_j) = \tau^{-1}\sigma^{-1}(a_j) = \tau^{-1}(a_{j-1}) = b_{j-1} $$
    Hence, $\sigma^{-1}(b_j) = b_{j-1}$. This proves that $(b_0, \dots, b_{k-1})$ is a cycle of $\sigma$. Finally, note that $(a_0, \dots, a_{k-1})$ and $(b_0, \dots, b_{k-1})$ satisfy the conclusion of the lemma for $i=0$. 
\end{proof}

The second lemma is an asymptotic estimate that we use when we upper bound the number of $(\sigma, \tau)$ pairs such that $S(\sigma, \tau)$ is multi-banded.

\begin{lemma}\label{lem:asymptotic}
    Let $n \in \NN$. Then,
$$\frac{1}{n!} \sum_{k=1}^{\floor{n/2} }\frac{(n-k)!}{k} = O\left(\frac{1}{n}\right)$$
    where $\floor{n/2}$ denotes the largest integer less than or equal to $n/2$.
\end{lemma}
\begin{proof}
Define $T = \floor{n^{1/3}}$. Note for $n$ large enough $T < \floor{n/2}$. Hence, we can split the sum as
\begin{equation}\label{eq:splitSum}\frac{1}{n!} \sum_{k=1}^{\floor{n/2}}\frac{(n-k)!}{k} = \sum_{k=1}^T \frac{(n-k)!}{n!k} + \sum_{k=T+1}^{\floor{n/2}} \frac{(n-k)!}{n!k}\end{equation}
and perform the asymptotic analysis separately on the terms to show 
$$\sum_{k=1}^T \frac{(n-k)!}{n!k} = O\left(\frac{1}{n}\right) \text{ and }  \sum_{k=T+1}^{\floor{n/2}} \frac{(n-k)!}{n!k} = o\left(\frac{1}{n}\right)$$
First note that 
\begin{equation}\label{eq:factorialEstimate}\frac{(n-k)!}{n!} = \frac{1}{n (n-1) \dots (n-k+1)} = \frac{1}{ n^k \prod_{j=1}^{k-1} \left(1 - \frac{j}{n}\right)}\end{equation}
Also note, 
\begin{equation}\label{eq:smallKLogError}\log \left(\prod_{j=1}^{k-1} \left(1 - \frac{j}{n}\right)\right) = \sum_{j=1}^{k-1} \log \left(1- \frac{j}{n}\right) = \sum_{j=1}^{k-1}\left(-\frac{j}{n} - \frac{j^2}{2n^2} + O\left(\frac{j^3}{n^3}\right)\right) = O\left(\frac{k^2}{n}\right)\end{equation}
where the second equality follows from the Taylor expansion of $\log(1-x)$ and the last equality follows by summing term-wise up to $k-1$ and noting the fact that $k \leq T = o(\sqrt{n})$. 

Exponentiating (\ref{eq:smallKLogError}) we see,
$$\prod_{j=1}^{k-1} \left(1 - \frac{j}{n}\right) = e^{O(k^2/n)} = 1 + O(k^2/n) + O(k^4/n^2)$$
where we've used the Taylor expansion of $e^x$ coupled with the fact that $k = o(\sqrt{n})$ here. Using this in (\ref{eq:factorialEstimate}), we see,
$$\frac{(n-k)!}{n!} = \frac{1}{n^k}\cdot\frac{1}{1+O(k^2/n)} = \frac{1}{n^k}\left(1+O\left(\frac{k^2}{n}\right)\right)$$
using the expansion of $\frac{1}{1+x} = 1 - x + O(x^2)$. Hence, the first term in (\ref{eq:splitSum}) becomes
\begin{equation}
    \sum_{k=1}^T \frac{(n-k)!}{n!k} = \sum_{k=1}^T \frac{1}{k n^k}\left(1+ O \left(\frac{k^2}{n}\right)\right) = \sum_{k=1}^T \frac{1}{k n^k}+ \sum_{k=1}^T \frac{1}{k n^k} O \left(\frac{k^2}{n}\right) = \sum_{k=1}^T \frac{1}{k n^k} + O\left(\frac{1}{n}\sum_{k=1}^T\frac{k}{n^k}\right)
\end{equation}
Note that $\sum_{k=1}^\infty \frac{k}{n^k} = \frac{n}{(n-1)^2} = O(n^{-1})$. Therefore, 
$$\sum_{k=1}^T \frac{(n-k)!}{n!k} = O\left(\frac{1}{n}\right) + O\left(\frac{1}{n^2}\right) = O\left(\frac{1}{n}\right)$$
which concludes the analysis for the first term in (\ref{eq:splitSum}).

For the second term we note that when $k > T$, $(n-k)! < (n-T)!$ so that,

\begin{equation}
   \sum_{k=T+1}^{\floor{n/2}} \frac{(n-k)!}{n!k} < \frac{(n-T)!}{n!}\sum_{k=T+1}^{\floor{n/2}}\frac{1}{k}   \leq \frac{1}{(n-T)^T} \cdot O(\log(n)).   
\end{equation}

Since $T = \floor{n^{1/3}} < n/2$, we have $(n-T)^T \geq (n/2)^T$. Hence, 
$$\frac{1}{(n-T)^T} \cdot O(\log(n)) \leq \frac{1}{n}O\left(\frac{2^T \log(n)}{n^{T-1}}\right) = o\left(\frac{1}{n}\right)$$
This concludes the analysis of the second term in (\ref{eq:splitSum}).\end{proof}

We are now ready to prove Theorem~\ref{thm:singleBanded}.

\begin{proof}[Proof of Theorem~\ref{thm:singleBanded}]

For a given $n$, we will crudely upper bound the proportion of $(\sigma, \tau) \in S_n \times S_n$ such that $S(\sigma, \tau)$ is multi-banded. We will then demonstrate that this upper bound tends to $0$ as $n \rightarrow \infty$. 

In light of Lemma~\ref{lem:multibandCharacterization}, we aim to upper bound the number of $(\sigma, \tau) \in S_n \times S_n$ such that $\sigma$ contains disjoint cycles $(a_0, \dots, a_{k-1})$ and $(b_0, \dots, b_{k-1})$ of the same length such that for some $i$, $\tau(b_{i+j}) = a_j$ for all $j = 0,\dots, k$, where the respective $a$ and $b$ indices are taken modulo $k$.

Towards this end, let $n \in \NN$ be given. Fix $k \in \{1, 2, \dots, \floor{n/2}\}$. First, the number of ways to pick two disjoint cycles $(a_0, \dots, a_{k-1})$ and $(b_0, \dots, b_{k-1})$ of length $k$ from $\{1, \dots, n\}$ is given by 
$$\frac{1}{2}\cdot  \frac{n!}{(n-2k)! k^2}$$
as there are $\frac{n!}{(n-k)!k}$ ways to pick the first cycle and $\frac{(n-k)!}{(n-2k)!k}$ ways to pick the second cycle but the order of which cycle gets picked first does not matter.

Then the number of permutations $\sigma$ with $(a_0, \dots, a_{k-1})$ and $(b_0, \dots, b_{k-1})$ as cycles is $(n-2k)!$. Finally, the number of permutations $\tau$ for which there exists $i$ such that $\tau(b_{i+j}) = a_j$ for $j = 0, \dots, k-1$ (and the indices of the $a$'s and $b$'s taken modulo $k$) is $k \cdot (n-k)!$

Summing these counts for $k = 1, \dots, \floor{n/2}$ and then dividing by the total number of pairs then gives us an upper bound on the proportion of pairs $(\sigma, \tau)\in S_n \times S_n$ such that $S(\sigma, \tau)$ has the multi-band property. Hence, for $(\sigma, \tau) \sim \Unif(S_n \times S_n)$,
\begin{align*}
    \Pr[S(\sigma, \tau) \text{ is multi-banded}] &\leq \frac{1}{|S_n \times S_n|} \cdot \sum_{k=1}^{\floor{n/2}}\frac{1}{2}\cdot \frac{n!}{(n-2k)! k^2}\cdot (n-2k)! \cdot k \cdot (n-k)!\\
    &= \frac{1}{2\cdot n!}\sum_{k=1}^{\floor{n/2}}\frac{(n-k)!}{k}\\
    &= O\left(\frac{1}{n}\right)
\end{align*}
where the last equality uses Lemma~\ref{lem:asymptotic}.
\end{proof}

\begin{remark}
In light of Theorem~\ref{thm:singleBanded}, we see that if we want to restrict the number of maximal horizontal cylinders in our standard model, then it suffices to restrict the number of cycles of the permutation $\sigma$, hence motivating the HR model that we subsequently analyze.

\end{remark}

\section{Convergence of induced distribution on $A_n$ to uniform distribution}\label{sec:convergenceofdist}

Recall from Section \ref{sec:randommodels} that given a random model that chooses two permutations from $S_n$ (and hence a random square-tiled surface), considering the commutator of the two permutations induces a distribution on $A_n$. It is a natural question to ask how different this distribution on $A_n$ is from the uniform distribution. In the case that it is not very different (in the sense of total variation distance), one can hope to use well known permutation statistics for the uniform distribution on $A_n$ and apply it to the distribution induced by our random model. 

Towards this end, in this section we prove one of our main theorems that the distribution on $A_n$ induced by the HR-model (under certain further restrictions), converges in total variation distance to the uniform distribution on $A_n$:

\begin{theorem}[Convergence theorem]\label{thm:probabilityConvergenceReal} Let $U_n$ be the uniform distribution on $A_n$. Let $\alpha \in [0, \frac{1}{2})$ and let $\mu_n \vdash n$ such that $\mu_n$ has at most $n^\alpha$ parts. Consider the HR model based on $\alpha$ and let $P_{\mu_n}$ be the corresponding probability distribution induced on $A_n$. Then, for any $\epsilon > 0$ such that $\alpha + \epsilon < \frac{1}{2}$, we have, 
$$||P_{\mu_n} - U_n|| = O\left(\frac{1}{n^{1-2\alpha-2\epsilon}}\right)$$ where $|| \cdot ||$ denotes max norm difference or total variation distance.
\end{theorem}

\begin{remark}
    \begin{enumerate}
        \item Once $\alpha$ is fixed, for a given $n \in \NN$, there might be multiple partitions $\mu_n \vdash n$ with at most $n^\alpha$ parts. The conclusion of the theorem holds regardless of the choice of partitions $\mu_n$ as long as the number of parts is at most $n^\alpha$. i.e. for any sequence of partitions $\mu_n \vdash n$ with at most $n^\alpha$, we get convergence of $P_{\mu_n}$ to $U_n$ in total variation distance. 
        \item When $\alpha = 0$,  $\mu_n = (n)$ so that the HR model consists of STSs with the horizontal permutation being restricted to $n$-cycles. In this case the error term becomes $O(n^{-1+2\epsilon})$ for any $\epsilon \in (0, 1/2)$. In fact, careful onsideration of the values of irreducible characters on the $n$-cycle conjugacy class can improve this error to be independent of $\epsilon$. See Theorem \ref{thm:onecylinderprobconvergence}.
    \end{enumerate}
\end{remark}

To prove this theorem, we start with a few lemmas that establish a formula for $P_{\mu_n}$  in terms of the irreducible characters of $S_n$. The first of these lemmas gives an expression for the number of conjugate pairs of permutations whose product is a fixed element of the alternating group:

\begin{lemma}\label{lem:NumberOfProductPairs}
Let $\mu_n \vdash n$ be a partition of $n$. Then, for any $g \in A_n$,

$$\big|\{(\sigma, \pi) \in \calK_{\mu_n} \times \calK_{\mu_n} \,|\, \sigma \pi = g\}\big| = \frac{|\calK_{\mu_n}|^2}{n!} \sum_{\chi \in \Irr(S_n)} \frac{|\chi(\calK_{\mu_n})|^2 \chi(g)}{\chi(1)}.$$
    
\end{lemma}
\begin{proof}

Let $\mu_n \vdash n$ and $g \in A_n$ be given. Let $\calK_1, \dots ,\calK_m$ be all the distinct conjugacy classes of $S_n$. Let $K_i = \sum_{\sigma \in \calK_i} \sigma \in Z(\CC[S_n])$ be the class sum over the conjugacy class $\calK_i$ (recall Definition~\ref{def:classSum}). Without loss of generality, let $g \in \calK_1$. Note that by definition, $g$ is a summand in the class sum $K_1$.

Next, let $K_{\mu_n}$ be the class sum over $\calK_{\mu_n}$ and consider the product $K_{\mu_n} \cdot K_{\mu_n} \in Z(\CC[S_n])$. We observe that the coefficient of $g$ in $K_{\mu_n} \cdot K_{\mu_n}$ is precisely the quantity, $\big|\{(\sigma, \pi) \in \calK_{\mu_n} \times \calK_{\mu_n}\,|\,\sigma \pi = g\}\big|$.

By Proposition~\ref{prop:classsumbasis}, $K_1, \dots, K_m$ forms a basis for $Z(\CC[S_n])$. Hence, we can write,
\begin{equation}\label{eq:ProductAsBasisClassSum} K_{\mu_n} \cdot K_{\mu_n} = \sum_{j=1}^m a_j K_j \end{equation}
where $a_j \in \ZZ_{\geq 0}$ (by Proposition~\ref{prop:classsumbasis}). Note that the coefficient of $g$ on the right hand side of (\ref{eq:ProductAsBasisClassSum}) is $a_1$. Equating coefficents on both sides, we see that 
$$a_1 =\big|\{(\sigma, \pi) \in \calK_{\mu_n} \times \calK_{\mu_n}\,|\,\sigma \pi = g\}\big| $$
and we aim to isolate this coefficient.

Towards this end, for a fixed $\chi \in \Irr(S_n)$, we recall the function $\omega_\chi$ from Definition~\ref{def:omegaFunction} and apply it to both sides of (\ref{eq:ProductAsBasisClassSum}). Using Proposition~\ref{prop:omegachi} we obtain,
$$\frac{|\chi(\calK_{\mu_n})|^2 |\calK_{\mu_n}|^2}{\chi(1)^2} = \sum_{j=1}^m a_j \frac{\chi(\calK_j)|\calK_j|}{\chi(1)}.$$
Now multiplying both sides by $\chi(1)\overline{\chi(\calK_1)}$ and summing over all irreducible characters we have,
\begin{align*}
    \sum_{\chi \in \Irr(S_n)}\frac{|\chi(\calK_{\mu_n})|^2 |\calK_{\mu_n}|^2\overline{\chi(\calK_1)}}{\chi(1)} = \sum_{\chi \in \Irr(S_n)}\sum_{j=1}^m a_j \chi(\calK_j)|\calK_j|\overline{\chi(\calK_1)} = \sum_{j=1}^m a_j |\calK_j|\sum_{\chi \in \Irr(S_n)}  \chi(\calK_j)\overline{\chi(\calK_1)} = a_1 |S_n|
\end{align*}
where, in the last step, we use second orthogonality of irreducible characters from Theorem~\ref{thm:characterprops}. Hence, noting that irreducible characters of $S_n$ take on real values,
$$\big|\{(\sigma, \pi) \in \calK_{\mu_n} \times \calK_{\mu_n}\,|\,\sigma \pi= g\}\big| =a_1 = \frac{|\calK_{\mu_n}|^2}{n!}\sum_{\chi \in \Irr(S_n)}\frac{|\chi(\calK_{\mu_n})|^2 \chi(\calK_1)}{\chi(1)} =  \frac{|\calK_{\mu_n}|^2}{n!}\sum_{\chi \in \Irr(S_n)}\frac{|\chi(\calK_{\mu_n})|^2 \chi(g)}{\chi(1)}.$$
\end{proof}

The next lemma now establishes a character-theoretic formula for $P_{\mu_n}$.

\begin{lemma}\label{lem:sigmaProbExpression}
     Let $\mu_n \vdash n$. Then for $g \in A_n$, \begin{equation}P_{\mu_n}(g) = \frac{1}{n!}\sum_{\chi \in \Irr(S_n)} \frac{|\chi(\calK_{\mu_n})|^2 \chi(g)}{\chi(1)}.\end{equation}
\end{lemma}
\begin{proof}
By definition,
\begin{equation}\label{eq:main}P_{\mu_n}(g) = \frac{\big|\{(\sigma, \tau) \in \calK_{\mu_n} \times S_n\,|\, \sigma \tau \sigma^{-1} \tau^{-1} = g \}\big|}{|\calK_{\mu_n}|\cdot |S_n|} =\frac{1}{|\calK_{\mu_n}|\cdot |S_n|} \sum_{\sigma \in \calK_{\mu_n}} \sum_{\substack{\tau \in S_n \\ \sigma \tau \sigma^{-1} \tau^{-1} = g}} 1.\end{equation}
For a fixed $\sigma \in \calK_{\mu_n}$, as $\tau$ ranges over all of $S_n$, $\tau \sigma^{-1} \tau^{-1}$ ranges over all elements in $\calK_{\mu_n}$. So, 
\begin{equation}\label{eq:commutatorToProduct}\sum_{\sigma \in \calK_{\mu_n}} \sum_{\substack{\tau \in S_n \\ \sigma \tau \sigma^{-1} \tau^{-1} = g}} 1 = \sum_{\sigma \in \calK_{\mu_n}} \sum_{\substack{\pi \in S_n \\ \sigma \pi = g}} \big|\{\tau \in S_n\,|\, \tau \sigma^{-1} \tau^{-1} = \pi\}\big|.\end{equation}
But note that given $\sigma^{-1}, \pi \in \calK_{\mu_n}$, there exists $\tau_0 \in S_n$ such that $\pi = \tau_0 \sigma^{-1} \tau_0^{-1}$. So, for any $\tau \in S_n$, we have that,
$$\tau \sigma \tau^{-1} = \pi \iff \tau_0^{-1} \tau \sigma^{-1} \tau^{-1}\tau_0 = \sigma^{-1}.$$
This is equivalent to $\tau_0^{-1} \tau$ being in the centralizer of $\sigma^{-1}$ in $S_n$ which we denote by $C_{S_n}(\sigma^{-1})$. This, in turn, is equivalent to $\tau \in \tau_0 C_{S_n}(\sigma^{-1}).$
Hence, \begin{equation}\label{eq:centralizer}\big|\{\tau \in S_n\,|\, \tau \sigma^{-1} \tau^{-1} = \pi\}\big| = |\tau_0 C_{S_n}(\sigma^{-1})| = |C_{S_n}(\sigma^{-1})| = \frac{|S_n|}{|\calK_{\mu_n}|}\end{equation}
where the last equality follows from the orbit-stabilizer lemma. 
Using (\ref{eq:commutatorToProduct}) and (\ref{eq:centralizer}), we rewrite (\ref{eq:main}) as,
\begin{equation}\label{eq:doubleSum}
    P_{\mu_n}(g) = \frac{1}{|\calK_{\mu_n}|\cdot |S_n|} \sum_{\sigma \in \calK_{\mu_n}} \sum_{\substack{\pi \in S_n \\ \sigma \pi = g}} \frac{|S_n|}{|\calK_{\mu_n}|} =\frac{1}{|\calK_{\mu_n}|^2} \sum_{\sigma \in \calK_{\mu_n}} \sum_{\substack{\pi \in S_n \\ \sigma \pi = g}} 1.
\end{equation}
From Lemma~\ref{lem:NumberOfProductPairs}, we then have
$$\sum_{\sigma \in \calK_{\mu_n}} \sum_{\substack{\pi \in S_n \\ \sigma \pi = g}} 1 = \frac{|\calK_{\mu_n}|^2}{n!} \sum_{\chi \in \Irr(S_n)} \frac{|\chi(\calK_{\mu_n})|^2 \chi(g)}{\chi(1)}.$$
Together with (\ref{eq:doubleSum}), we conclude
$P_{\mu_n}(g) = \frac{1}{n!}\sum_{\chi \in \Irr(S_n)} \frac{|\chi(\calK_{\mu_n})|^2 \chi(g)}{\chi(1)}.$\end{proof}

Our next lemma gives an upper bound for the total variation distance of the probability distributions $P_{\mu_n}$ and $U_n$. This is a variation of the famous Diaconis-Shahshahani Upper Bound Lemma (Lemma 14 from \cite{DiaShah}): 

\begin{lemma}\label{lem:sigmaTowardUniform}
Let $\mu_n \vdash n$ and $\sigma \in \calK_{\mu_n}$. Let $U_n$ be the uniform distribution on $A_n.$ Then, 
\begin{equation*}||P_{\mu_n} - U_n||^2 \leq \frac{1}{4} |A_n|\sum_{g \in A_n}|P_{\mu_n}(g) - U_n(g)|^2\leq \frac{1}{4} \sum_{\lambda \neq (n), \lambda \neq (1^n)} \frac{1}{(\dim \rho^\lambda)^2} |\chi^\lambda(\sigma)|^4.\end{equation*} 
\end{lemma}
\begin{proof}
    Since $P_{\mu_n}$ is a class function, for $n \geq 5,$ by Lemma \ref{lem:ds}, we have that \begin{equation}\hat{P}_{\mu_n}(\rho) = \frac{1}{\dim \rho} \sum_{\mathcal{K}_i} P_{\mu_n}(\mathcal{K}_i) |\mathcal{K}_i| \chi^\rho(\mathcal{K}_i) I_{\dim \rho},\end{equation} where $I_{\dim \rho}$ is the identity matrix of dimension $\dim \rho$, the sum is taken over the conjugacy classes $\mathcal{K}_i$ of $S_n,$ and $P_{\mu_n}(\mathcal{K}_i)$ is $P_{\mu_n}$ evaluated on an arbitrary element of this conjugacy class. By equation \eqref{eq:pNDef}, we have that \begin{align*}\hat{P}_{\mu_n}(\rho) & = \frac{1}{n! \dim \rho} \sum_{\chi \in \Irr(S_n)} \frac{|\chi(\sigma)|^2}{\chi(1)} \sum_{\mathcal{K}_i} \chi(\mathcal{K}_i) \chi^\rho (\mathcal{K}_i) |\mathcal{K}_i| I_{\dim \rho} \\ & = \frac{I_{\dim \rho}}{n! \dim \rho} \sum_{\chi \in \Irr(S_n)} \frac{|\chi(\sigma)|^2}{\chi(1)} \sum_{ \tau \in S_n} \chi(\tau) \chi^\rho(\tau) \\ & = \frac{I_{\dim \rho}}{(\dim \rho)^2} |\chi^\rho(\sigma)|^2.\end{align*}
    Now note that 
    $$||P_{\mu_n} - U_n||^2 = \left(\frac{1}{2}\sum_{g \in A_n}|P_{\mu_n}(g) - U_n(g)|\right)^2 \leq \frac{1}{4} |A_n| \sum_{g \in A_n}|P_{\mu_n}(g) - U_n(g)|^2$$
    where the last inequality follows from the Cauchy-Schwartz inequality. Then, from Diaconis-Shashahani \cite{DiaShah} as well as Chmutov-Pittel \cite{ChmuPit} (Equation 2.10), we have that \begin{equation}\label{eqn:fouriertransformbound}\frac{1}{4} |A_n| \sum_{g \in A_n}|P_{\mu_n}(g) - U_n(g)|^2 \leq \frac{1}{4}\sum_{\lambda \neq (n), \lambda \neq (1^n)} \dim \rho^\lambda \text{Tr}\left[\hat{P}_{\mu_n} (\rho^\lambda) \hat{P}_{\mu_n}(\rho^\lambda)^*\right],\end{equation} for $n \geq 5.$ Therefore, $$||P_{\mu_n} - U_{n}||^2  \leq \frac{1}{4} \sum_{\lambda \neq (n), \lambda \neq (1^n)} \frac{1}{(\dim \rho^\lambda)^2} |\chi^\lambda(\sigma)|^4,$$ as required.     \end{proof}

With these Lemmas in hand, we are ready to prove Theorem \ref{thm:probabilityConvergenceReal}

\begin{proof}[Proof of Theorem \ref{thm:probabilityConvergenceReal}] Let $\alpha \in [0, \frac{1}{2})$ be fixed and let $\epsilon > 0$ be given such that $\alpha + \epsilon < \frac{1}{2}$. By Theorem \ref{thm:larsenshalevbound}, there exists some $N \in \mathbb{N}$ such that for all $n > N,$ $\chi(1)^{\alpha + \epsilon} \geq |\chi(\sigma)|,$ where $\sigma$ is an arbitrary element of $\calK_{\mu_n}.$ Hence, by Lemma \ref{lem:sigmaTowardUniform}, we have that \begin{align*}||P_{\mu_n} - U_n||^2 & \leq \frac{1}{4} \sum_{\lambda \neq (n), \lambda \neq (1^n)} \frac{1}{(\dim \rho^\lambda)^2} |\chi^\lambda(\sigma)|^4 \\ & \leq \frac{1}{4} \sum_{\lambda \neq (n), \lambda \neq (1^n)} \frac{1}{(\dim \rho^\lambda)^2} |\chi^\lambda(1)|^{4\alpha + 4\epsilon} \\ & = \frac{1}{4}\sum_{\lambda \neq (n), \lambda \neq (1^n)} \frac{1}{(\dim \rho^\lambda)^{2-4\alpha - 4\epsilon}} \\ 
& = \frac{1}{4}\left(-2 + \sum_{\lambda \vdash n} \frac{1}{(\dim \rho^\lambda)^{2-4\alpha - 4\epsilon}}\right) \\
& = O\left(\frac{1}{n^{2-4\alpha - 4\epsilon}}\right), \end{align*} where the second-to-last step follows from the observation that the characters corresponding to $(1^n)$ and $(n)$ are one-dimensional and the last step follows from Theorem \ref{thm:liebeckshalev} since $2 - 4\alpha - 4 \epsilon > 0$.\end{proof}

In the above proof of Theorem \ref{thm:probabilityConvergenceReal}, we used a bound from Larsen-Shalev on character ratios (Theorem \ref{thm:larsenshalevbound}) which led to the bound of $O(n^{1-2\alpha-2\epsilon})$. However, for a given $\mu_n \vdash n$ and the corresponding HR-model, specifically analysing the values of irreducible characters on the conjugacy class corresponding to $\mu_n$ can yield better bounds.

For instance, when $\alpha = 0$, $\mu_n$ is the partition corresponding to the conjugacy class of $n$-cycles. In this case, we can improve the error bound to $O(n^{-1})$ convergence. For this, we first need a lemma on the character values on $n$-cycles:

\begin{lemma}\label{lem:onecylindercharacter}
Let $\chi \in \Irr(S_n)$ be an irreducible character of $S_n$. If $\mu_n = (n)$ (the partition of $n$ with just one part) then $\chi (\calK_{\mu_n}) \in \{-1,0, 1\}$. 
\end{lemma}
\begin{proof}
This is a straightforward application of Theorem \ref{thm:murnaghannakayama}, the Murnaghan-Nakayama Rule. 

Given an irreducible character of $S_n$, let $\lambda \vdash n$ be the partition of $n$ which indexes the character.

To evaluate this character on an $n$-cycle using the Murnaghan-Nakayama Rule, we consider the border-strip tableau of shape $\lambda$ and type $\mu_n=(n)$. In particular, we consider ways to fill the Young diagram associated to $\lambda$ with $n$ 1's. There is precisely one way to do so and the result may or may not be a border-strip tableau. If the result is not a border-strip tableau, $\chi^\lambda(\calK_{\mu_n}) = 0$. If the result is a border-strip tableau, there is exactly one term in the sum in Equation \eqref{eq:murnaghannakayama}, hence $\chi^\lambda(\calK_{\mu_n}) = \pm 1$ in this case. \end{proof}

With this, we get the following slightly improved result:
\begin{theorem}\label{thm:onecylinderprobconvergence}
Let $U_n$ be the uniform distribution on $A_n$ and let $\mu_n = (n)$ be the partition of $n$ with a exactly one part and consider the probability distribution, $P_{\mu_n}$ on $A_n$ induced by the HR model with $\mu_n$. Then,
$$||P_{\mu_n} - U_n|| = O(n^{-1})$$
\end{theorem}
\begin{proof} The proof is very similar to the proof of Theorem \ref{thm:probabilityConvergenceReal}. By Lemmas \ref{lem:sigmaTowardUniform} and \ref{lem:onecylindercharacter}, we have
\begin{align*}||P_{\mu_n} - U_n||^2 & \leq \frac{1}{4} \sum_{\lambda \neq (n), \lambda \neq (1^n)} \frac{1}{(\dim \rho^\lambda)^2} |\chi^\lambda(\sigma)|^4 \\ 
& \leq \frac{1}{4} \sum_{\lambda \neq (n), \lambda \neq (1^n)} \frac{1}{(\dim \rho^\lambda)^2} \\ 
& = O\left(\frac{1}{n^2}\right), \end{align*} where the last step follows as in Theorem \ref{thm:probabilityConvergenceReal}
\end{proof}

\begin{remark}
    Via a similar (albeit more complicated) combinatorial analysis, we note that the above Lemma \ref{lem:onecylindercharacter}
    holds, for large enough $n$, for conjugacy classes with cycle type $\mu_n = (n-1,1), (n-2,2),(n-3,3),(n-2,1,1),(n-3,2,1)$. Hence, Theorem \ref{thm:onecylinderprobconvergence} holds if we replace the $n$-cycle with any of these conjugacy classes. 
\end{remark}

\section{Moments of vertex count}\label{sec:moments}

We devote this section to proving the following version of Theorem 2.4 from \cite{ShresRandom} (initially motivated by Fleming-Pippenger \cite{FlemPip}), which gives an approximation for the moments of the number of vertices of a random STS by the moments of the number of cycles of $\pi \sim \Unif(A_n)$:

\moments* 

Recall that $(\sigma, \tau) \sim \Unif(\calK_{\mu_n} \times S_n)$ induces a probability distribution $P_{\mu_n}$ on $A_n$ via considering the commutator $[\sigma, \tau]$. Let $C_{\mu_n}$ be the number of cycles of of a $P_{\mu_n}$-distributed permutation from $A_n$. In light of Proposition \ref{prop:commutatorstratum}, we see that $V_{\mu_n} = C_{\mu_n}$. Hence we work with $C_{\mu_n}$ for the remainder of this section. 

We start by deriving a form for the probability generating function of $C_{\mu_n}$ in the following Proposition. The proof is inspired by the proof of Lemma 4.5 in \cite{ShresRandom}:

\begin{proposition}[Probability Generating Function for Number of Cycles]\label{prop:ProbGenFcn}
Fix $\mu_n \vdash n$ and let $C_{\mu_n}$ be the number of cycles of $[\sigma, \tau]$ for $(\sigma, \tau) \sim \Unif(\calK_{\mu_n} \times S_n)$. 
The probability generating function of $C_{\mu_n}$  is given by 
$$g(q):= \sum_{x = 1}^n\Pr(C_{\mu_n} = x) q^x = \frac{1}{n!}\sum_{\lambda \vdash n}  |\chi^\lambda (\calK_{\mu_n})|^2 \prod_{r \in \lambda} (q+\cont(r))$$
where the product is over all squares in the Young diagram of $\lambda \vdash n$.
\end{proposition}
\begin{proof}
For any $\pi \in S_n$, define
$$N_{comm} (\pi) = \#\{(\sigma, \tau) \in \calK_{\mu_n} \times S_n\,|\, [\sigma, \tau] = \pi\} $$
which counts the number of pairs whose commutator gives $\pi$. By definition, we know $P_{\mu_n}(\pi) = \frac{N_{comm}(\pi)}{|\calK_{\mu_n} \times S_n|}$. So, using Lemma \ref{lem:sigmaProbExpression} and noting that irreducible characters of $S_n$ are indexed by partitions $\lambda \vdash n$, we obtain the following expression for $N_{comm}$:

\begin{equation}\label{eq:Ncommexpression}N_{comm}(\pi) = n! \cdot |\calK_{\mu_n}| \cdot \frac{1}{n!}\sum_{\lambda \vdash n} \frac{|\chi^\lambda (\calK_{\mu_n})|^2 \chi^\lambda(\pi)}{\chi^\lambda(1)} = |\calK_{\mu_n}| \sum_{\lambda \vdash n} \frac{|\chi^\lambda (\calK_{\mu_n})|^2 \chi^\lambda(\pi)}{\chi^\lambda(1)}
\end{equation}

Note that $N_{comm}$ is a class function. So, we can apply the Frobenius characteristic map (Equation \eqref{eq:frobeniuscharmap1}), to get,
\begin{align}\label{eq:chNcomm1}
\ch N_{comm} = \frac{1}{n!}\sum_{\pi \in S_n} N_{comm}(\pi )p_{\cyc(\pi)} 
= \frac{1}{n!} \sum_{\substack{\sigma \in \calK_{\mu_n}\\ \tau \in S_n}} p_{\cyc([\sigma, \tau])}
\end{align}

Alternately, using Equation \eqref{eq:frobeniuscharmap2} and the expression of $N_{comm}$ from Equation \eqref{eq:Ncommexpression} we get that
\begin{align*}
    \ch N_{comm} &= \sum_{\lambda \vdash n} \ideal{ N_{comm}, \chi^\lambda} s_\lambda\\
    &= \sum_{\lambda \vdash n} \left\langle|\calK_{\mu_n}| \sum_{\nu \vdash n} \frac{|\chi^\nu (\calK_{\mu_n})|^2 \chi^\nu}{\chi^\nu(1)}, \chi^\lambda\right\rangle s_\lambda\\
    &=\sum_{\lambda \vdash n} |\calK_{\mu_n}| \sum_{\nu \vdash n} \left(\frac{|\chi^\nu (\calK_{\mu_n})|^2}{\chi^\nu(1)}\ideal{\chi^\nu, \chi^\lambda}\right)s_\lambda
\end{align*}
Using orthogonality of irreducible characters of $S_n$ we obtain,
\begin{equation}\ch N_{comm}=\sum_{\lambda \vdash n} |\calK_{\mu_n}|  \frac{|\chi^\lambda (\calK_{\mu_n})|^2}{\chi^\lambda(1)}s_\lambda \label{eq:chNcomm2}
\end{equation}
Equation the two expressions from Equations \eqref{eq:chNcomm1} and \eqref{eq:chNcomm2}, we get
\begin{equation}\label{eq:twoexpressions}
\frac{1}{n!} \sum_{\substack{\sigma \in \calK_{\mu_n}\\ \tau \in S_n}} p_{\cyc([\sigma, \tau])} = \sum_{\lambda \vdash n} |\calK_{\mu_n}|  \frac{|\chi^\lambda (\calK_{\mu_n})|^2}{\chi^\lambda(1)}s_\lambda 
\end{equation}

Now for a given natural number $q$, let indeterminates $x_1 = 1, x_2 = 1, \dots, x_q = 1$ and $x_j =0$ for $j > q$. Then for $\lambda =(\lambda_1, \dots, \lambda_k) \vdash n$, we have that the power sum symmetric functions evaluate to, 
\begin{align*}
  &p_0 = 1 \\
  &p_m = \sum_{i=1}^\infty x_i^m = q\\
  &p_\lambda = p_{\lambda_1}\cdot p_{\lambda_2} \dots p_{\lambda_k}= q \cdot q  \dots q = q^k 
\end{align*}
where $k$ is the number of parts of $\lambda$ and also equivalently the number of cycles of the conjugacy class of $S_n$ with cycle type $\lambda$. Then, setting  indeterminates $x_1 = 1, x_2 = 1, \dots, x_q = 1$ and $x_j =0$ for $j > q$ and using Lemma \ref{lem:schurspec}, the above Equation \eqref{eq:twoexpressions} becomes,
\begin{align*}
    \sum_{\substack{\sigma \in \calK_{\mu_n} \\ \tau \in S_n}} q^{C_{\mu_n}(\sigma, \tau)} &= (n!)|\calK_{\mu_n}|  \sum_{\lambda \vdash n}   \frac{|\chi^\lambda (\calK_{\mu_n})|^2}{\chi^\lambda(1)} \prod_{r \in \lambda} \frac{q + \cont(r)}{h(r)}\\
    &= (n!)|\calK_{\mu_n}|  \sum_{\lambda \vdash n}  \frac{|\chi^\lambda (\calK_{\mu_n})|^2}{\chi^\lambda(1)} \frac{\chi^\lambda(1)}{n!} \prod_{r \in \lambda}(q + \cont(r))\\
    &= |\calK_{\mu_n}|  \sum_{\lambda \vdash n}  |\chi^\lambda (\calK_{\mu_n})|^2 \prod_{r \in \lambda}(q + \cont(r))
\end{align*}
Finally, noting that $$g(q) = \sum_{\substack{\sigma \in \calK_{\mu_n} \\ \tau \in S_n}} \frac{q^{C_{\mu_n}(\sigma, \tau)}}{(n!)|\calK_{\mu_n}|}$$ we obtain the desired result. \end{proof}

Using this, we get the following estimate on the tail of the distribution of $C_{\mu_n}$:

\begin{lemma}[Large deviations for number of vertices]\label{lem:largedeviations} Fix $\alpha \leq 1$ and $\mu_n \vdash n$ with at most $n^\alpha$ parts. 
Let $C_{\mu_n}$ be the number of cycles of a random permutation of $A_n$ distributed according to $P_{\mu_n}$. Then, for any $\epsilon > 0$, we have
$$ \Pr(C_{\mu_n} \geq t) = O\left(\frac{n^{\frac{5}{2} - \alpha - \epsilon}}{2^{t+n(1-2\alpha-2\epsilon)}}\right) $$
\end{lemma}

\begin{proof}
    By definition
    \begin{align*}
        \Pr(C_{\mu_n} \geq t) &= \sum_{s \geq t} \Pr(C_{\mu_n} = s) \\
        &\leq \sum_{s \geq t}\Pr(C_{\mu_n} = s) q^{s-t} \text{ when } q \geq 1\\
        &=\frac{1}{q^t} \sum_{s \geq t}\Pr(C_{\mu_n} = s) q^{s}\\
        &\leq \frac{g(q)}{q^t}
    \end{align*}

where $g(q):= \sum_{s=1}^\infty \Pr(C_{\mu_n} = s)q^s$ is the probability generating function for the random variable $C_{\mu_n}$. From Proposition \ref{prop:ProbGenFcn} we know that 
\begin{equation}\label{eqn:probgen}g(q) = \frac{1}{n!}\sum_{\lambda \vdash n}  |\chi^\lambda (\calK_{\mu_n})|^2 \prod_{r \in \lambda} (q+\cont(r))\end{equation}
The inequality $\Pr(C_{\mu_n} \geq t) \leq \frac{g(q)}{q^t}$ is valid for $q \geq 1$. In particular, we have, when $q = 2$, $\Pr(C_{\mu_n} \geq t) \leq \frac{g(2)}{2^t}$. However, when $q=2$, the product $\prod_{r \in \lambda} (q+\cont(r))$ is zero when there exists a square in the Young diagram of $\lambda$ with content -2. Hence, the only non-zero terms in the sum in (\ref{eqn:probgen}) correspond to partitions $\lambda$ of the form $(n-k, k)$. So,
\begin{align*}
    \Pr(C_{\mu_n} \geq t) &\leq \frac{1}{2^t n!}\sum_{k = 0}^{\floor{n/2}} |\chi^{(n-k,k)}(\mu_n)|^2 \prod_{r \in (n-k,k)}(2+ \cont(r)) 
\end{align*}
Now, by Theorem \ref{thm:larsenshalevbound}, for a given $\epsilon > 0$, there exists $N \in \NN$ such that for all $n > N$, we have $|\chi^{(n-k,k)}(\mu_n)|^2 \leq |\chi^{(n-k,k)}(1)|^{2 \alpha + 2\epsilon}$. Hence, for such $n > N$,
$$\Pr(C_{\mu_n} \geq t)\leq \frac{1}{2^t n!}\sum_{k = 0}^{\floor{n/2}} |\chi^{(n-k,k)}(1)|^{2\alpha + 2\epsilon} \prod_{r \in (n-k,k)}(2+ \cont(r)) $$
Finally, using Proposition \ref{prop:dimensionbound} we have
\begin{align*}
    \Pr(C_{\mu_n} \geq t) &\leq \frac{1}{2^t n!}\sum_{k = 0}^{\floor{n/2}} {n \choose k}^{2\alpha + 2\epsilon} \prod_{r \in (n-k,k)}(2+ \cont(r)) \\
    &= \frac{1}{2^t}\sum_{k = 0}^{\floor{n/2}} {n \choose k}^{2 \alpha + 2\epsilon} \frac{(n-k+1)! k!}{n!}\\
    &= \frac{1}{2^t}\sum_{k = 0}^{\floor{n/2}} {n \choose k}^{2 \alpha + 2\epsilon-1} (n-k+1)
\end{align*}
Note that ${n \choose k}$ is maximized when $k = \floor{n/2}$. By Stirling's approximation that $n! \sim \sqrt{2 \pi n} \left(\frac{n}{e}\right)^n$, we have for $\delta = 2 \alpha + 2\epsilon - 1$,
$${n \choose \floor{n/2}}^{\delta} = O\left(\frac{2^{\delta n}}{n^{\delta/2}}\right). $$
Since $(n-k+1) = O(n)$ when $k = \floor{n/2}$, each term in the sum has an asymptotic bound of $O\left( 2^{\delta n}\cdot n^{1-\delta/2}\right)$. Since there are $O(n)$ terms in the sum, we finally get
$$\Pr(C_{\mu_n} \geq t) = O( 2^{\delta n - t} \cdot n^{2 - \delta/2} ) = O\left(\frac{n^{\frac{5}{2}-\alpha -\epsilon}}{2^{t + n(1-2\alpha - 2\epsilon)}} \right).$$
\end{proof}

We are finally ready to prove Theorem \ref{thm:moments}.

\begin{proof}[Proof of Theorem \ref{thm:moments}]
First note that 
$$ \EE[p(C_{\mu_n})] = \sum_{s = 0}^n p(s)\Pr(C_{\mu_n} = s) = p(0) + \sum_{s=1}^n(p(s) - p(s-1))\Pr(C_{\mu_n} \geq s)$$
Similarly, we have
$$\EE[p(C_n)] = \sum_{s = 0}^n p(s)\Pr(C_n = s) = p(0) + \sum_{s=1}^n(p(s) - p(s-1))\Pr(C_n \geq s)$$
Subtracting the two expressions we get
\begin{align*} &\EE[p(C_{\mu_n})] - \EE[p(C_n)]\\ =& \sum_{s=1}^n(p(s) - p(s-1))(\Pr(C_{\mu_n} \geq s) - \Pr(C_n \geq s))\\
=& \sum_{s=1}^t (p(s) - p(s-1))(\Pr(C_{\mu_n} \geq s) - \Pr(C_n \geq s))\\ &+ \sum_{s=t+1}^n(p(s) - p(s-1))(\Pr(C_{\mu_n} \geq s) - \Pr(C_n \geq s)\\
\leq&\sum_{s=1}^t (p(s) - p(s-1))|\Pr(C_{\mu_n} \geq s) - \Pr(C_n \geq s)|\\ &+ \sum_{s=t+1}^n(p(s) - p(s-1))\Pr(C_{\mu_n} \geq s) + \sum_{s=t+1}^n(p(s) - p(s-1)) \Pr(C_n \geq s)\\
\leq & p(t) \cdot || P_{\mu_n} - U_n||+ p(n) \cdot \Pr(C_{\mu_n} \geq t) + p(n)\cdot  \Pr(C_n\geq t)
\end{align*}
Finally, using Theorem \ref{thm:probabilityConvergenceReal}, Lemma \ref{lem:largedeviations} and Proposition \ref{lem:largedevforuniform}, we get that for $\alpha  \in [0, \frac{1}{2})$ fixed and any $\epsilon, \epsilon_2 > 0$ with $\alpha + \epsilon < \frac{1}{2}$,
$$\EE[p(C_{\mu_n})] - \EE[p(C_n)] = O\left(\frac{t^l}{n^{1-2\alpha - 2\epsilon}}\right)  + O\left(\frac{n^{l+\frac{5}{2}-\alpha -\epsilon_2}}{2^{t + n(1-2\alpha - 2\epsilon_2)}} \right) + O\left(\frac{n^{l+1}}{2^t}\right)$$
When $\alpha < \frac{1}{2}$, we can choose $\epsilon_2$ such that $\alpha + \epsilon_2 = \frac{1}{2}$. Then, we have
$$\EE[p(C_{\mu_n})] - \EE[p(C_n)] = O\left(\frac{t^l}{n^{1-2\alpha - 2\epsilon}}\right)  + O\left(\frac{n^{l+2}}{2^t} \right) + O\left(\frac{n^{l+1}}{2^t}\right)$$

Taking $t = \frac{(l + 3-2\alpha-2\epsilon)\log(n)}{\log(2)}$, we ensure that $O\left(\frac{t^l}{n^{1-2\alpha-2\epsilon}}\right)$ asymptotically dominates the other terms on the right side of the inequality to give the desired result. \end{proof}

\section{Topological Statistics: Connectedness, Genus Distribution and most likely Stratum} \label{sec:topologicalstatistics}
\subsection{Connectedness}

In both the Standard Model ($S_n \times S_n$) and the HR Model ($\calK_{\mu_n} \times S_n)$ (when $\mu_n \neq (n)$), it is possible that a random pair of permutations gives a disconnected surface with $n$ squares. For instance, if $(\sigma, \tau) = (\Id, \Id) \in S_n \times S_n$, we see that $S(\sigma, \tau)$ corresponds to a union of $n$ disjoint square-tori. Likewise if $(\sigma, \tau) = ((1,\cdots, n-1)(n), \Id) \in \calK_{\mu_n} \times S_n$ with $\mu_n$ being the conjugacy class consisting of $(n-1)$ cycles, the STS $S(\sigma, \tau)$ is a disjoint union of two tori -- one with a single square and the other with $n-1$ squares.

However, we will show in Proposition \ref{prop:connectedness} that asymptotically (as $n\rightarrow \infty)$, the probability of obtaining a disconnected surface decays to zero in both the Standard and HR models. For this we will utilize a classical result of Dixon \cite{Dix} and a more recent generalization due to Babai-Hayes which we state below. Recall that given $G \leq S_n$ a permutation group, an integer $x \in \{1, \dots, n\}$ is a fixed point of $G$ if $g(x) = x$ for all $g \in G$.

\begin{theorem}[Babai-Hayes \cite{BabHay}]\label{thm:babaihayes2} Let $G \leq S_n$ be a permutation group with $f \leq n/2$ fixed points and let $\tau \in S_n$ be chosen uniformly randomly. Then, 
$$\Pr(\ideal{G, \tau} \text{ is transitive }) > 1 - (f+1)\left(\frac{1}{n} + O(n^{-2})\right)$$
\end{theorem}

\begin{proposition}[Asymptotic Connectedness]\label{prop:connectedness} The probability of obtaining a connected surface in both the Standard and HR Models is asymptotically 1. More precisely, we have the following. 
\begin{enumerate}
\item (Standard Model) Let $(\sigma, \tau)$ be a uniform random pair of permutations from $S_n \times S_n$. Then,
$$ \Pr(S(\sigma, \tau) \text{ is connected}) = 1 - \frac{1}{n} +O\left(\frac{1}{n^2}\right) \text{ as }n\rightarrow \infty $$

\item (HR Model) Fix $\alpha \in [0,1)$. Let $(\sigma, \tau)$ be a uniform random pair of permutations from $\calK_{\mu_n} \times S_n$ where $\mu_n \vdash n$ is a partition of $n$ with at most $n^\alpha$ parts. Then,
$$ \Pr(S(\sigma, \tau) \text{ is connected})  > 1 - \frac{1}{n^{1-\alpha}} - O\left(\frac{1}{n^{2-\alpha}}\right) \text{ as }n\rightarrow \infty$$

\end{enumerate}

\end{proposition}

\begin{proof}

An STS with horizontal and vertical permutations $\sigma$ and $\tau$ is connected if and only if the subgroup $\ideal{\sigma, \tau} \subseteq S_n$ generated by $\sigma$ and $\tau$ is transitive. 
\begin{enumerate}
\item (Standard Model) This follows directly from Dixon (Theorem 2 of \cite{Dix}) who proved that the probability a uniformly random pair of permutations in $S_n \times S_n$ generates a transitive subgroup in $S_n$ is $1 - n^{-1} + O(n^{-2})$.

\item (HR Model) For a fixed $\alpha < 1$, note that there exists $N \in \mathbb{N}$ such that for all $n \geq N$, $n^\alpha < \frac{n}{2}$. Now, let $n \geq N$. For any $G = \ideal{\sigma}$ such that $\sigma$ has at most $n^\alpha$ cycles, the number of fixed points of $G$ is less than $n^\alpha$. Applying Theorem \ref{thm:babaihayes2}, we see that for $(\sigma, \tau) \sim \Unif(\calK_{\mu_n} \times S_n)$, 
\begin{align*}
   \Pr(S(\sigma, \tau) \text{ is connected}) &=  \Pr(\ideal{\sigma, \tau} \text{ is transitive}) \\
   &> 1-n^\alpha \left( \frac{1}{n} + O(n^{-2})\right)\\
   &= 1 - \frac{1}{n^{1-\alpha}} - O\left(\frac{1}{n^{2-\alpha}}\right).\end{align*}
\end{enumerate} \end{proof}

\subsection{Genus distribution}
In Theorem \ref{thm:moments}, we obtained the moments for the number of vertices of a random square-tiled surface in the HR-model using Theorem \ref{thm:probabilityConvergenceReal}. As a Corollary (using Proposition \ref{prop:expectednumcycles}), we get the expectation and variance of the genus:

\expectedgenus*
\begin{proof}
Fix $\alpha \in [0, \frac{1}{2})$ and let $\mu_n$ be a partition of $n$ with at most $n^\alpha$ parts.
Let $V_{\mu_n}$ be the number of vertices of $S(\sigma, \tau)$ for $(\sigma,\tau) \sim \Unif(\calK_{\mu_n} \times S_n)$ and $C_n$ be the number of cycles of $\sigma$ for $\sigma \sim \Unif(A_n)$. 
By Theorem~\ref{thm:moments} and Proposition~\ref{prop:expectednumcycles}, for any $\epsilon > 0$ such that $\alpha + \epsilon < \frac{1}{2}$, we have, 
$$\EE[V_{\mu_n}] = \log(n) + \gamma + o(1) +O\left(\frac{\log(n)}{n^{1-2\alpha - 2\epsilon}}\right) = \log(n) + \gamma + o(1).$$
Recall that $G_{\mu_n} = \frac{n}{2} - \frac{V_{\mu_n}}{2}+1$. So, 
$$\EE[G_{\mu_n}] = \frac{n}{2} - \frac{\EE[V_{\mu_n}]}{2}+1 = \frac{n}{2}- \frac{\log(n)}{2}- \frac{\gamma}{2}+o(1).$$
Likewise, 
\begin{align*}
    \Var[G_{\mu_n}] = \frac{1}{4}\Var[V_{\mu_n}] = \frac{1}{4}(\EE[V_{\mu_n}^2] - \EE[V_{\mu_n}]^2) &= \frac{1}{4}\left(\EE[C_n^2] - \EE[C_n]^2 -2 \EE[C_n] \cdot O\left(\frac{\log(n)}{n^{1-2\alpha - 2\epsilon}} \right)+o(1)\right)\\
    &= \frac{1}{4}\Var[C_n] + o(1)\\
    &= \frac{\log(n)}{4} + \frac{\gamma}{4}- \frac{\pi^2}{24} + o(1).
\end{align*}
    
\end{proof}
As a consequence of Theorem \ref{thm:moments} we can also obtain the asymptotic distribution of the genus. Let $K_n$ be the number of cycles of $\pi \sim \Unif(S_n)$ and recall that $C_n$ is the number of cycles of $\pi \sim \Unif(A_n)$.  From Chmutov-Pittel \cite{ChmuPit} we first have a strong central local limit theorem for $C_n$ as follows:

\begin{proposition}[\cite{ChmuPit}]
Let $a > 0$ fixed.  Then uniformly for $\ell$ satisfying $\frac{\ell - \EE[K_n] }{\sqrt{\Var[K_n]}} \in [-a,a]$, 
$$\Pr(C_n = \ell) = \frac{(2 + O(\Var[K_n]^{-1/2} )e^{-\frac{1}{2} \cdot \frac{(\ell - \EE[K_n])^2}{ \Var[K_n]}}}{\sqrt{2 \pi \Var[K_n]}} $$
    
\end{proposition}

It is known that $\EE[K_n] = \log(n) + O(1)$ and $\Var[K_n] = \log(n) + O(1)$ (See Proposition \ref{prop:expectednumcycles}). Since for any fixed $\alpha \in [0, \frac{1}{2})$, one can choose $\epsilon > 0$ such that $\frac{1}{n^{1-2\alpha-2\epsilon}} \in O\left(\frac{1}{\Var[K_n]}\right) = O\left(\frac{1}{\log(n)}\right)$, using Theorem \ref{thm:probabilityConvergenceReal} we get the following distribution of the number of vertices:
\begin{theorem}
Let $\alpha \in [0, \frac{1}{2})$ and $\mu_n \vdash n$ with at most $n^\alpha$ parts. Let $a > 0$ fixed. Then uniformly for $\ell$ satisfying $\frac{\ell - \EE[K_n] }{\sqrt{\Var[K_n]}} \in [-a,a]$, 
$$\Pr(V_{\mu_n} = \ell) = \frac{(2 + O(\log(n)^{-1/2} )e^{-\frac{1}{2} \cdot \frac{(\ell - \EE[K_n])^2}{ \Var[K_n]}}}{\sqrt{2 \pi \Var[K_n]}} $$

\end{theorem}

Finally, since the genus, $G_{\mu_n}$, is related to the number of vertices, $V_{\mu_n}$, by the formula:
$$G_{\mu_n} = \frac{n}{2} - \frac{V_{\mu_n}}{2} +1,$$ we obtain the distribution of $G_{\mu_n}$:

\genusdistribution*

\subsection{Likelihood of strata and most likely stratum}

Consider any random model for choosing a labelled STS with $n$ squares which in turn gives two permutations $\sigma$ and $\tau$ in $S_n$. Since the commutator $[\sigma, \tau]$ is always an element of $A_n$ just as in the case for the Standard and HR models, choosing a random labelled STS induces a distribution on $A_n$. Let $P_n$ be this distribution. 

We devote this section to proving the following theorem which gives a sufficient criterion on the random model that ensures that the most likely stratum is the one with a single large cone point:

\sufficientcriterionmostlikelystratum*

Before stating the proof, we need to analyze the relative sizes of conjugacy classes of $S_n.$ In particular, using Proposition \ref{prop:cycleNumbers} we will find the three largest conjugacy classes in $S_n$. Towards that, we have the following intermediary lemma which states that, in most cases, combining the largest two cycles of a conjugacy class produces another conjugacy class that is larger or equal in size. 

\begin{lemma}\label{lem:coalescecycles} Let $n \geq 4$.
    Let $\calK$ be a conjugacy class in $S_n$ with canonical representative $(1 \cdots m_1) (m_1 + 1 \cdots m_2) (m_2 + 1 \cdots m_3) \cdots (m_{k-1}+1 \cdots m_{k}).$ Let $\calK'$ be the conjugacy class in $S_n$ with representative $(1 \cdots m_2)(m_2+1 \cdots m_3) \cdots (m_{k-1}+1 \cdots m_k)$. Suppose $m_1 \geq 2$ and $m_2-m_1 \geq 2$ (i.e. $\calK$ has at least two non-trivial cycles). Then, $|\calK'| \geq |\calK|$. 
\end{lemma}
\begin{proof}
    Let $(a_1, \dots, a_n)$ be the cycle type of $\calK$ and consider $|\calK'| - |\calK|$. There are four cases depending on the relative lengths of the largest cycles of $\calK$:

    Case 1: $m_1 > m_2 - m_1 > m_3 - m_2$. In this case, by Proposition \ref{prop:cycleNumbers}, \begin{equation*}|\mathcal{K}'| - |\mathcal{K}| = \left(\prod_{k = 1}^{m_3 - m_2} \frac{k^{-a_k}}{a_k !}\right)\left( \frac{1}{m_2} - \frac{1}{m_1 (m_2 - m_1)}\right).\end{equation*} 

    Now note that given two integers $a, b \geq 2$, we have that $(a-1)(b-1) - 1 \geq 0$. Hence, $ab \geq a+b$. Applying this fact, we see that $|\mathcal{K}'| - |\mathcal{K}|\geq 0$. 

    Case 2: $m_1 > m_2 - m_1$ and $m_2 - m_1 = m_3 - m_2.$ Again, \begin{align*}|\mathcal{K}'| - |\mathcal{K}| & = \left(\prod_{k = 1}^{m_2 - m_1} \frac{k^{-a_k}}{a_k !}\right) \left(\frac{a_{m_2 - m_1} (m_2 - m_1)}{m_2}  - \frac{1}{m_1} \right) \\ & = a_{m_2 - m_1}(m_2 - m_1) \left(\prod_{k = 1}^{m_2 - m_1} \frac{k^{-a_k}}{a_k !}\right) \left(\frac{1}{m_2}  - \frac{1}{a_{m_2 - m_1} m_1 (m_2 - m_1)} \right).\end{align*} Now, the term $\frac{1}{a_{m_2 - m_1} m_1 (m_2 - m_1)}$ is smaller than $\frac{1}{m_1(m_2 - m_1)},$ so by the argument in Case 1, we have that $|\mathcal{K}'| - |\mathcal{K}|$ is positive in this case.

    Case 3: $m_1 = m_2 - m_1> m_3 - m_2.$ Then, \begin{align*}|\mathcal{K}'| - |\mathcal{K}| & = \left(\prod_{k = 1}^{m_3 - m_2} \frac{k^{-a_k}}{a_k !}\right) \left(\frac{1}{m_2}  - \frac{1}{2 m_1(m_2 - m_1)} \right),\end{align*} which again, by the argument in Case 1, is positive.

    Case 4: $m_1 = m_2 - m_1 = m_3 - m_2.$ We then have that \begin{align*}|\mathcal{K}'| - |\mathcal{K}| & = \left(\prod_{k = 1}^{m_1} \frac{k^{-a_k}}{a_k !}\right) \left(\frac{a_{m_1} (a_{m_1}-1) m_1 (m_2 - m_1) }{m_2}  - 1 \right).\end{align*} Since $a_{m_1} \geq 3,$ the above is positive by the same argument as in Case 1.

We note that in Cases 1 and 3, when $m_3 - m_2 = 0$, the product $\left(\prod_{k = 1}^{m_3 - m_2} \frac{k^{-a_k}}{a_k !}\right)$ above is taken to be 1 so that the result of those cases still hold. 

    Since the cases were exhaustive, $\mathcal{K}'$ has at least as many elements as $\mathcal{K}.$    \end{proof}
Next, we identify the three largest conjugacy classes in $S_n$:

\begin{lemma}
    The three largest conjugacy classes in $S_n$ for $n \geq 7$ are the classes of $(1 \cdots n-1)(n),$ $(1 \cdots n),$ and $(1 \cdots n-3)(n-2 \,\,\, n-1)(n)$ in that order. \label{lem:classesInSN}
\end{lemma}

\begin{proof}
Let $\calK_1, \calK_2$ and  $\calK_3$ be the conjugacy classes of $(1 \cdots n-1)(n),$ $(1 \cdots n),$ and $(1 \cdots n-3)(n-2 \,\,\, n-1)(n)$ respectively. 

We start by considering $\calK_3$. Note that the conjugacy class of the identity consists of only 1 element so it is smaller than $\calK_3$. We deal with the rest of the conjugacy class in a few cases:

Case 1: Let $\calK$ be the conjugacy class with a single $k$-cycle and $(n-k)$ 1-cycles where $k \geq 2$ and $(n-k) \geq 2$. Then,
$$|\calK| = \frac{n!}{k(n-k)!} \,\text{ and }\, |\calK_3| = \frac{n!}{2(n-3)}$$
Note that we have $k(n-k)! \geq k(n-k)$ which is minimized when $k = 2$ and $k = n-2$ so that $k(n-k)! \geq 2(n-2) > 2(n-3)$. Hence $|\calK_3| \geq |\calK|$.

Case 2: Let $\calK$ be a conjugacy class with at least two non-trivial cycles. We break this into further cases:

Case 2.1: If $\calK$ has exactly two cycles (both non-trivial), then the canonical representative has the form $(1 \cdots k)(k+1 \cdots n)$ with $k \geq 2$ and $n-k \geq 2$. If $k = n-k$ we have $|\calK| = \frac{n!}{2k^2}$. Noting that $2k^2 = 2k(n-k) \geq 2n > 2(n-3)$, we conclude that $|\calK_3| > |\calK|$. If $k > n-k$ we have  $|\calK| = \frac{n!}{k(n-k)}$. Note that $k(n-k)$ with $2 \leq k \leq n-2$ is minimized when $k = 2$ and $k = n-2$. Hence, $k(n-k) \geq 2(n-2) > 2(n-3)$ implying $|\calK_3| > |\calK|$.

Case 2.2: If $\calK$ has exactly three cycles and all are non-trivial, then a canonical representative has the form $(1 \cdots k)(k+1 \cdots j)(j+1 \cdots n)$ with $k \geq 2$, $j-k \geq 2$ and $n-j \geq 2$. Applying Lemma \ref{lem:coalescecycles}, we see that $\calK$ is smaller than the class of $(1 \cdots j)(j+1 \cdots n)$. Applying Case 2.1, this is smaller than the class of $\calK_3$.

Case 2.3: If $\calK$ has exactly three cycles and exactly two are non-trivial, then the canonical representative has the form $(1 \cdots k)(k+1 \cdots n-1)(n)$ with $k \geq 2$ and $n-1-k \geq 2$. If $k = n-1-k$, then $|\calK| = \frac{n!}{2k^2}$. Noting that $2k^2 = 2k(n-1-k) \geq 2(n-1) > n-3$ we conclude $|\calK_3| > |\calK|$.
If $k > (n-1-k)$, then $|\calK| =\frac{n!}{k(n-1-k)}$. Noting that $k(n-1-k)$ is minimized when $k = 2$ and $k = n-3$, we have $k(n-1-k) \geq 2(n-3)$. Hence, $|\calK_3| \geq |\calK|$.

Case 2.4: If $\calK$ has more than three cycles and more than one are non-trivial, then applying Lemma \ref{lem:coalescecycles} repeatedly, can reduce to Case 1 (when $\calK$ has more than one trivial cycle), Case 2.3 (when $\calK$ has exactly one trivial cycle) and Case 2.1 (when $\calK$ has no trivial cycles).

This shows that $\calK_3$ is larger or equal in size than any conjugacy class $\calK$ with at least two non-trivial cycles or with one non-trivial cycle and at least two trivial cycles. There are two other conjugacy classes left: one with exactly one non-trivial cycle and no trivial cycles and another with exactly one non-trivial cycle and exactly one trivial cycle. 

For the former, if a conjugacy class $\calK$ has exactly one non-trivial cycle and no trivial cycles, then the canonical representative is $(1 \cdots n)$ and $\calK = \calK_2$. In that case 
$$|\calK_2| = \frac{n!}{n}.$$
Since $2(n-3) > n$ when $n \geq 7$, we see that $|\calK_2| > |\calK_3|.$ For the latter, if a conjugacy class $\calK$ has exactly one non-trivial cycle and exactly one trivial cycle, then the canonical representative is $(1 \cdots n-1)(n)$ and $\calK = \calK_1$. In that case $$|\calK_1| = \frac{n!}{n-1}.$$ Since $n > n-1$, $|\calK_1| > |\calK_2|$. 
Therefore, $|\calK_1| > |\calK_2| > |\calK_3|$. \end{proof}

The next Lemma gives a bound on the difference in probability (of a uniformly randomly chosen permutation in $A_n$) between landing in the largest conjugacy class of $S_n$ contained in $A_n$ and any other conjugacy class. 

\begin{lemma}
    Let $n \geq 7$, let $\calK_{\max}$ be the largest conjugacy class of $S_n$ contained in $A_n$.  Let $\calK$ be any other conjugacy class of $S_n$ inside $A_n$. For $\pi \sim \Unif(A_n)$,
    \begin{equation*}\Pr(\pi \in \calK_{\max}) - \Pr(\pi \in \calK) \geq \frac{n - 6}{n(n-3)}.\end{equation*} \label{lem:uniformBound} 
\end{lemma}

\begin{proof}
Depending on the parity of $n$, either $(1 \cdots n)$ or $(1 \cdots n-1)(n)$ is contained in $A_n.$ By Lemma \ref{lem:classesInSN}, this means that $\Pr(\pi \in \calK_{\max}) \geq \frac{2}{n}.$ Moreover, by Lemma \ref{lem:classesInSN}, we have that $\text{Pr}(\pi \in \calK) \leq \frac{2}{n!}|\calK_{(1 \cdots n-3)(n-2 \,\, n-1)(n)}| = \frac{1}{n-3}$ where $\calK_{(1 \cdots n-3)(n-2 \,\, n-1)(n)}$ is the conjugacy class of $S_n$ containing $(1 \cdots n-3)(n-2 \,\, n-1)(n)$. Hence, 
    $$\Pr(\pi \in \calK_{\max}) - \Pr(\pi \in \calK) \geq \frac{2}{n} - \frac{1}{n-3}  = \frac{n-6}{n(n-3)}.$$\end{proof}

We are now able to deduce asymptotic bounds on the probabilities of landing in a given conjugacy class with the $P_n$ distribution on $A_n$:

\begin{theorem}
    Let $g \in A_n$ with cycle type $t = (a_1, a_2, \ldots, a_n)$, and let $\mathcal{K}_g$ be the conjugacy class of $g$ in $S_n.$ Let $P_n$ be a probability distribution on $A_n$ induced by a random square-tiled surface model such that equation \eqref{eqn:l2bound} holds for some $s \in \RR$. Assume also that $P_n$ is constant on conjugacy classes of $S_n$. Let $\sigma$ be a random permutation from $A_n$ distributed according to $P_n$. Then, \begin{equation*}\left|\Pr(\sigma \in \mathcal{K}_g) - 2 \prod_{k = 1}^n \frac{k^{-a_k}}{a_k !} \right|= O\left(\frac{1}{n^{(1 + s)/2}}\right) \end{equation*}
    with the constant in the asymptotic upper bound independent of the element $g \in A_n$.
    \label{thm:stratumprobabilities} 
\end{theorem}

\begin{proof}

    Let $U_n$ be the uniform distribution on $A_n$. Let $\calK_\text{max}$ be the conjugacy class of $(1 \cdots n-1)(n)$ if $n$ is even and $(1 \cdots n)$ if $n$ is odd which is the largest conjugacy class of $S_n$ in $A_n$. By Proposition \ref{prop:cycleNumbers} and noting that both $P_n$ and $U_n$ are class functions, we have that \begin{align*}\left|\text{Pr}(\sigma \in \mathcal{K}_g) - 2 \prod_{k = 1}^n \frac{k^{-a_k}}{a_k !}\right| & = \left | \sum_{h \in \mathcal{K}_g} (P_n(h) - U_n(h)) \right| \\
    & = |\mathcal{K}_g||P_n(g) - U_n(g)| \\
    & = \sqrt{|\mathcal{K}_g|} \sqrt{|\mathcal{K}_g||P_n(g) - U_n(g)|^2} \\ 
    & = \sqrt{|\mathcal{K}_g|}\sqrt{\sum_{h \in \mathcal{K}_g}|P_n(h) - U_n(h)|^2} \\ 
    & \leq  \sqrt{\frac{|\calK_\text{max}|}{|A_n|}} \sqrt{|A_n|\sum_{h \in A_n} |P_n(h) - U_n(h)|^2} = O\left(\frac{1}{n^{(1+s)/2}}\right),\end{align*} 
    since $\frac{|\mathcal{K}_\text{max}|}{|A_n|} = O\left(\frac{1}{n}\right)$ with constant independent of $g$.
\end{proof}

The error term is large enough that for most strata the given upper bound on probabilities are meaningless. However, these errors are independent of stratum, and this will be enough to give us the result we desire about the largest stratum, Theorem \ref{thm:sufficientcriterion}.

\begin{proof}[Proof of Theorem \ref{thm:sufficientcriterion}] Define $\calH_n = \calH(n-1)$ if $n$ is odd and $\calH_n = \calH(n-2)$ if $n$ is even and let $\calK_n$ be the conjugacy class corresponding to stratum $\calH_n$. Let $\mathcal{H}_n'$ be a stratum with $\mathcal{H}_n' \neq \mathcal{H}_n$ and let $\mathcal{K}_n'$ be the conjugacy class corresponding to $\mathcal{H}'_n.$ 

Let $\sigma \in A_n$ be a random permutation distributed according to $P_n$ and let $\pi \sim \Unif(A_n)$.  By Theorem \ref{thm:stratumprobabilities}, there exists some positive $M \in  \mathbb{R}$ and some $N_1 \in \mathbb{N}$ such that for all $n \geq N_1$ and $g \in A_n,$ we have that \begin{equation*}M > n^{(1 + s)/2} \left|\text{Pr}(\sigma \in \mathcal{K}_g) - \Pr(\pi \in \calK_g)\right|\end{equation*} 
where $\calK_g$ is the conjugacy class of $S_n$ that contains $g$.
Since $s >1$, there exists $N_2 \in \mathbb{N}$ be such that for all $n \geq N_2$, \begin{equation*}\frac{n^{(1 + s)/2} (n - 6)}{2n(n-3)} > M.\end{equation*}
Then for all $n \geq \max(N_1, N_2),$
\begin{align*}\Pr(\sigma \in \mathcal{K}_n) - \text{Pr}(\sigma \in \mathcal{K}_n')  = & \text{Pr}(\sigma \in \mathcal{K}_n) - \text{Pr}(\pi \in \mathcal{K}_n) \\
& +  \text{Pr}(\pi \in \mathcal{K}_n) - \text{Pr}(\pi \in \mathcal{K}_n') \\
& + \text{Pr}(\pi \in \mathcal{K}_n') -\text{Pr}(\sigma \in \mathcal{K}_n') \\ 
 \geq & - \left|\text{Pr}(\sigma \in \mathcal{K}_n) - \text{Pr}(\pi \in \mathcal{K}_n)\right|\\
 & + \text{Pr}(\pi \in \mathcal{K}_n) - \text{Pr}(\pi \in \mathcal{K}_n') \\
 & -  \left|\text{Pr}(\pi \in \mathcal{K}_n') -\text{Pr}(\sigma \in \mathcal{K}_n')\right|\\
\geq & - \frac{2 M}{n^{(1 + s)/2}} + \frac{n-6}{n(n-3)} > 0,\end{align*} where we use Lemma \ref{lem:uniformBound} to conclude that $\text{Pr}(\pi \in \mathcal{K}_n) - \text{Pr}(\pi \in \mathcal{K}_n') \geq \frac{n-6}{n(n-3)}.$ Thus, for all $n \geq \max(N_1, N_2),$ we have that $\mathcal{H}_n$ is the most probable stratum in the model. \end{proof}

As a corollary to Theorem \ref{thm:sufficientcriterion}, we deduce the most likely stratum for both the standard and HR-models:

\mostlikelystratum*

\begin{proof} For the standard model, we note from the proof of Lemma 2.1 of \cite{ShresRandom} that the standard model satisfies equation \eqref{eqn:l2bound} with $s = 2$. Hence, the result follows from Theorem \ref{thm:sufficientcriterion}.

For the HR-model, we fix $\alpha \in [0, \frac{1}{4})$ and let $\mu_n \vdash n$ be a partition with at most $n^\alpha$ parts. From Lemma \ref{lem:sigmaTowardUniform} and following the proof of Theorem \ref{thm:probabilityConvergenceReal}, for any $\epsilon > 0$ with $\alpha + \epsilon < \frac{1}{4}$, we have 
$$\frac{1}{4} |A_n|\sum_{g \in A_n}|P_{\mu_n}(g) - U_n(g)|^2\leq \frac{1}{4} \sum_{\lambda \neq (n), \lambda \neq (1^n)} \frac{1}{(\dim \rho^\lambda)^2} |\chi^\lambda(\sigma)|^4 =  O\left(\frac{1}{n^{2-4\alpha-4\epsilon}}\right).$$
Now note that $P_{\mu_n}$ satisfies equation \eqref{eqn:l2bound} with $s = 2 - 4\alpha - 4\epsilon > 1.$ Hence, by Theorem \ref{thm:sufficientcriterion}, the most likely strata in the HR-model based on $\alpha$ is either $\calH(n-1)$ or $\calH(n-2)$.
\end{proof}

\begin{remark} Note that a weaker version of Theorem \ref{thm:stratumprobabilities} is immediate from Theorem \ref{thm:probabilityConvergenceReal}. In particular, if $P_n$ satisfies equation \eqref{eqn:l2bound}, then it satisfies $$||P_n - U_n||^2 = O(n^{-s}),$$ so that immediately we obtain 
    $$\left|\Pr(\sigma \in \mathcal{K}_g) - 2 \prod_{k = 1}^n \frac{k^{-a_k}}{a_k !} \right| \leq ||P_n - U_n|| =   O\left(\frac{1}{n^{s/2}}\right)$$
However, this weaker bound forces $s > 2$ in Theorem \ref{thm:sufficientcriterion}, which is not enough to give Corollary \ref{cor:mostlikelyStrataspecific}.

\end{remark}

\section{Holonomy theorem} \label{sec:holonomy}

We devote this section to proving the Holonomy Theorem:

\holonomy*

The proof of this theorem follows the strategy used in the Holonomy Theorem from \cite{ShresRandom} where an analogous theorem for the standard model is proved. In using the HR model, the bounds on the error terms are weakened.

\begin{proof}[Proof of Theorem \ref{thm:holonomy}]
 Motivated by Lemma \ref{lem:combcriterionforgeom} we consider the commutator $[\sigma, \tau]$ for both cases.
Define first sets
        $$X_n = \{ (\sigma, \tau) \in \calK_{\mu_n} \times S_n|\, [\sigma, \tau] \text{ is a derangement}\} \text{ and}$$
        $$Y_n = \{ (\sigma, \tau) \in \calK_{\mu_n} \times S_n| \,\emptyset \neq \text{ fixed points of }[\sigma, \tau] \subseteq \text{ fixed points of }\sigma  \cap  \text{ fixed points of }\tau \} $$
        Note that $X_n$ and $Y_n$ are disjoint sets. 
        By Lemma \hyperref[lem:holcriterion]{\ref*{lem:combcriterionforgeom}.\ref*{lem:holcriterion}}, 
$$\Pr(S(\sigma, \tau) \text{ is a holonomy torus}) = \Pr((\sigma, \tau) \in X_n) + \Pr((\sigma,\tau) \in Y_n).
$$
Let $\pi \sim \Unif(A_n)$. Then,
\begin{align*}
    &\left|\Pr(S(\sigma, \tau) \text{ is a holonomy torus})-\frac{1}{e} \right| \\
    =&  \left|\Pr((\sigma, \tau) \in X_n) + \Pr((\sigma,\tau) \in Y_n)-\frac{1}{e} \right| \\
    \leq& \left|\Pr((\sigma, \tau) \in X_n) - \Pr(\pi \text{ is a derangement})\right| + \left| \Pr(\pi \text{ is a derangement}) - \frac{1}{e}\right|+ |\Pr((\sigma,\tau) \in Y_n)|
\end{align*}
Note that every pair of permutations in $Y_n$ generates an non-transitive subgroup of $S_n$ since $\sigma$ and $\tau$ share fixed points. Hence, using 
Proposition \ref{prop:AnDerangement}, Theorem \ref{thm:probabilityConvergenceReal} and Proposition \ref{prop:connectedness}, we have, for any $\epsilon > 0$ such that $\alpha + \epsilon < \frac{1}{2}$,
\begin{align*}
    &\left|\Pr(S(\sigma, \tau) \text{ is a holonomy torus})-\frac{1}{e} \right| \\
    \leq\,& ||P_{\mu_n} - U_n|| + \frac{n^2}{(n+1)!} + \Pr(\ideal{\sigma, \tau} \text{ is not transitive})\\
    =\,& O\left(\frac{1}{n^{1-2\alpha-2\epsilon}} \right)+ \frac{n^2}{(n+1)!} + \frac{1}{n^{1-\alpha}} + O\left(\frac{1}{n^{2-\alpha}}\right)\\ 
    =\,& O\left(\frac{1}{n^{1-2\alpha-2\epsilon}} \right).
\end{align*}

Next, let $f_{\mu_n}$ be the number of fixed points of $[\sigma, \tau]$ with $(\sigma, \tau) \sim \Unif(\calK_{\mu_n} \times S_n)$. Recall that $C_{\mu_n}$ is the number of cycles of $[\sigma, \tau]$.  Using Lemma \hyperref[lem:viscriterion]{\ref*{lem:combcriterionforgeom}.\ref*{lem:viscriterion}},
\begin{align*}
    \Pr(S(\sigma, \tau) \text{ is a visibility STS}) &> \Pr\left(f_{\mu_n}  < \frac{n}{2}\right)\\ 
    &> \Pr\left(C_{\mu_n} < \frac{n}{2}\right)\\
    &=1 - \Pr\left(C_{\mu_n} \geq \frac{n}{2}\right)\\
    &= 1 - O\left(\frac{n^2}{2^{n/2}}\right)
\end{align*}
where the last equality follows from Lemma $\ref{lem:largedeviations}$.\end{proof}

\newpage
\bibliographystyle{abbrv}
\bibliography{references}

\end{document}